\documentclass[11pt,a4paper,reqno]{amsart}
\usepackage{amsfonts}
\usepackage{amsthm}
\usepackage{amsmath}
\usepackage{amscd}
\usepackage{accents}
\usepackage[latin2]{inputenc}
\usepackage{t1enc}
\usepackage[mathscr]{eucal}
\usepackage{indentfirst}
\usepackage{graphicx}
\usepackage{graphics}
\usepackage{pict2e}
\usepackage{epic}
\usepackage{stmaryrd}
\usepackage{authblk}
\usepackage{subcaption}
\usepackage[symbol]{footmisc}
\numberwithin{equation}{section}
\usepackage[margin=2.9cm]{geometry}
\usepackage{epstopdf} 
\RequirePackage{doi}
\usepackage{hyperref}
\allowdisplaybreaks
\usepackage{cite}

\usepackage{verbatim,amssymb,amsfonts}
\usepackage{mathrsfs}
\usepackage{mathtools}
\usepackage{abraces}
\usepackage{tikz-cd}
\usepackage{indentfirst}
\usepackage{slashed}
\usepackage{thmtools}
\usepackage{setspace}
\usepackage{enumerate}

\theoremstyle{plain}
\newtheorem{theorem}{Theorem}[section]
\newtheorem{lemma}[theorem]{Lemma}
\newtheorem{corollary}[theorem]{Corollary}
\newtheorem{proposition}[theorem]{Proposition}

\newtheorem{assumption}{Assumption}

 \theoremstyle{definition}
\newtheorem{definition}[theorem]{Definition}
\newtheorem{conjecture}[theorem]{Conjecture}
\newtheorem{remark}[theorem]{Remark}
\newtheorem{example}{Example}[section]
\newcommand{\im}{\operatorname{im}}

\DeclarePairedDelimiterX{\inp}[2]{\langle}{\rangle}{#1, #2}

\newcommand{\cB}{{\mathcal B}}
\newcommand{\cC}{{\mathcal C}}

\newcommand{\cE}{{\mathcal E}}
\newcommand{\cF}{{\mathcal F}}

\newcommand{\cH}{{\mathcal H}}

\newcommand{\cJ}{{\mathcal J}}
\newcommand{\cK}{{\mathcal K}}
\newcommand{\cL}{{\mathcal L}}
\newcommand{\cM}{{\mathcal M}}

\newcommand{\cP}{{\mathcal P}}
\newcommand{\cQ}{{\mathcal Q}}

\newcommand{\cS}{{\mathcal S}}

\newcommand{\cU}{{\mathcal U}}

\newcommand{\cW}{{\mathcal W}}

\newcommand{\fL}{{\mathfrak L}}

\newcommand{\id}{{\mathrm{Id} }}
\newcommand{\Gr}{{\mathrm{Gr} }}
 
\newcommand{\spec}{{\text{spec}}}

\newcommand{\spinc}{{ \text{spin}^c }}
\newcommand{\diag}{{ \mathrm{Diag} }}
\newcommand{\ind}{{ \mathrm{Ind} }}

\newcommand{\tr}{{ \mathrm{Tr} }}

\newcommand{\spf}{{ \mathrm{sf} }}
\newcommand{\mas}{{ \mathrm{Mas} }}
\newcommand{\Det}{{ \mathrm{det} }}

\newcommand{\ba}{\begin{eqnarray}}
\newcommand{\na}{\end{eqnarray}}
\newcommand{\ban}{\begin{eqnarray*}}
\newcommand{\nan}{\end{eqnarray*}}

\newcommand{\C}{{\mathbb C}}
\newcommand{\E}{{\mathbb E}}

\newcommand{\R}{{\mathbb R}}
\newcommand{\Z}{{\mathbb Z}}

\renewcommand{\thefootnote}{\fnsymbol{footnote}}

\makeatletter
\g@addto@macro{\endabstract}{\@setabstract}
\newcommand{\authorfootnotes}{\renewcommand\thefootnote{\@fnsymbol\c@footnote}}%
\makeatother

\makeatletter
\@namedef{subjclassname@2020}{%
  \textup{2020} Mathematics Subject Classification}
\makeatother

\title[]{Equivariant Spectral Flow and Equivariant $\eta$-invariants on Manifolds with Boundary} 



\begin{document}


    \maketitle 
    
\begin{center}
	
	\normalsize
	\authorfootnotes
	Johnny Lim\textsuperscript{1} \footnote[2]{Corresponding author.}, 
	Hang Wang\textsuperscript{2} 
	\par \bigskip
	
	\textsuperscript{1} \small{School of Mathematical Sciences, Universiti Sains Malaysia, Malaysia} \par
	\textsuperscript{2}School of Mathematical Sciences and Shanghai Key Laboratory of PMMP \\
East China Normal University, Shanghai 200241, P.R.China\par \bigskip
	
\end{center}

\email{johnny.lim@usm.my}
\email{wanghang@math.ecnu.edu.cn}


\begin{abstract}
In this article, we study several closely related invariants associated to Dirac operators on odd-dimensional manifolds with boundary with an action of the compact group $H$ of isometries. In particular, the equality between equivariant winding numbers, equivariant spectral flow, and equivariant Maslov indices is established. We also study equivariant $\eta$-invariants which play a fundamental role in the equivariant analog of Getzler's spectral flow formula. As a consequence, we establish a relation between equivariant $\eta$-invariants and equivariant Maslov triple indices in the splitting of manifolds. 
\end{abstract}



\section{Introduction}

Back in the 1970s, in a series of their seminal papers \cite{1,2,3}, Atiyah, Patodi, and Singer (APS) established a well-known index theorem for Dirac operators on even-dimensional compact manifolds with boundary, for which, if compared with the case of closed manifolds (cf. \cite{4}), a \textit{boundary correction term} arises that comprises the $\eta$-invariant and the kernel of some tangential operator. Not long after, Donnelly \cite{12b} extended the APS index theorem to the case of compact group actions. About three decades later, Dai and Zhang (DZ) in \cite{11}  provided a candidate for the counterpart of the APS index theorem where they considered a certain variant of Toeplitz operators on odd-dimensional manifolds with even-dimensional boundary, and proved an index formula taking the form $\mathrm{Ind}(T_g)= \displaystyle \int \alpha - \bar{\eta}_{DZ} +\tau^\mu(\cdot,\cdot,\cdot),$ where $\alpha$ is some local index forms, $\bar{\eta}_{DZ}$ is an even-type $\eta$-invariant specifically defined for an even-dimensional boundary, and $\tau^\mu$ is some Maslov triple index (cf. \cite{15}). It is then natural to ask about an equivariant analog of DZ, or in other words, an odd analog of Donnelly's equivariant index theorem. In \cite{16}, we carried out the study and established an \textit{equivariant Toeplitz index theorem}, extending DZ's work to the case of compact group actions. In particular, it is necessary to define new notions, including equivariant DZ $\eta$-invariants, equivariant spectral flows, equivariant Maslov triple indices, etc. 
This leads to the main theme of this article: to establish the equivariant extension of some classical invariants, which will serve as a prequel to \cite{16}. 
We emphasize that the main reference for the non-equivariant objects of study here is the work by Kirk and Lesch in \cite{15}. They elegantly showed the interlinked connections between $\eta$-invariants, spectral flows, Maslov (double and triple) indices, $\zeta$-determinants, and others.  

In this article, firstly, we study an equivariant version of the classical \textit{spectral flow} first introduced by APS in \cite{3}. For our purposes, we adopt a useful machinery by Phillips \cite{22} for an equivalent description of spectral flows in terms of grid partitions. 
The upshot is that the complication of counting the crossings is simplified by merely considering the intersection of a spectral path with the vertical ``wall'' of each partitioned grid. 
This is then generalised to the equivariant case  (Definition \ref{equivspecflowdef}) where it takes value in trace when evaluated over a grid partition. 
Under this setting, the equivariant spectral flow is well-defined (Proposition \ref{spfwelld}) and is an isomorphism (Theorem \ref{specisom}) between the $H$-equivariant fundamental group of the non-trivial component of the algebra of self-adjoint equivariant Fredholm operators and the representation ring $R(H).$


On the other hand, Getzler \cite{13} shows that the spectral flow of a path can be expressed in terms of APS's $\eta$-invariants \cite{1,2}, which comprises the evaluation at the endpoints and the variation of $\eta$-invariants along the path. Kirk and Lesch \cite{15} provide another proof of this  (in terms of reduced $\eta$-invariants) via the path-lifting method to the universal cover. In this paper, we generalise this well-known formula to the equivariant case in Theorem \ref{equivGetSpecFlow}. Moreover, we also establish the gluing formula (Theorem \ref{etasplitthm}) of equivariant $\eta$-invariants along the boundary splitting via a closed hypersurface. Readers can refer to \cite{8,20} for expositions on gluing formulae for spectral flow and $\eta$-invariants.

Another spectral invariants closely related to  $\eta$-invariants are the $\zeta$-function $\zeta(D^2,s)$ of the square of Dirac operators and its variation $\dot{\zeta}(D^2,s)$ evaluated at $s=0,$ with $\dot{\zeta}(D^2,0)$ being the most sensitive one in the sense that its value changes upon a perturbation by a smoothing operator, cf.  \cite{28}. We discuss an equivariant generalisation of both of these invariants explicitly in Section \ref{sec7}, as well as defining the equivariant $\zeta$-determinant of $D$ (Definition \ref{zetadet3}). Inspired by \cite{15}, we extend our discussion of $\zeta$-determinants to the equivariant Scott-Wojciechowski theorem (Conjecture \ref{equivsctwoj}) -- we merely state this as a conjecture, instead of a theorem, as it is not our intention to re-prove \cite{24} in full details in the equivariant settings. Despite this, we reasonably justify why our settings are compatible and extend to theirs.

In addition, we also show an identification (Theorem \ref{specmas}) between the equivariant spectral flow of a path of Dirac operators and the Maslov index of a pair of possibly infinite-dimensional Lagrangian $H$-spaces associated to the Calder\'on projection and boundary conditions. As an intermediate step, we give a description of the winding number (Definition \ref{equivwinding}) of equivariant unitary paths associated to a Lagrangian pair, which corresponds to a pair of boundary conditions. These can be viewed as slight generalisations of \cite{15}, which is itself an infinite version of the classical Maslov index, e.g. in \cite{21,9}.   

As a summary, this paper is dedicated in discussing the following correspondences and providing a unified overview of ``spectral intersection counting'' via different perspectives: 
\[
\substack{\text{\normalsize equivariant} \\ \text{\normalsize spectral flow}} \quad \longleftrightarrow  \quad
\substack{\text{\normalsize equivariant} \\ \text{\normalsize Maslov index}}
\quad \longleftrightarrow  \quad
\substack{\text{\normalsize equivariant} \\ \text{ \normalsize winding number}} 
\] 
where the corresponding ingredients are respectively
\vspace{0.2cm}
\[
\substack{\text{\normalsize an equivariant path $D_P(t)$} \\ \text{\normalsize connecting $P(t)$ and $\cP_M(t)$}} \; \leftrightarrow  \;
\substack{\text{\normalsize a pair of $H$-Lagrangians} \\ \text{\normalsize $(\cL_P(t),\cL_{\cP_M}(t))$}}
\; \leftrightarrow  \;
\substack{\text{\normalsize an equivariant path of} \\ \text{ \normalsize unitary operators $T^*(t)K(t).$} }
\]

The organisation of this paper is as follows. In Section \ref{sec2}, we lay out the basic setups, assumptions, and notations. In Section \ref{sec3}, we study the equivariant spectral flow of self-adjoint equivariant Fredholm operators via grid partitions. In Section \ref{sec4}, we study the equivariant Maslov index of pairs of $h$-Lagrangians. In Section \ref{sec5}, we discuss the equivariant winding number of some associated unitary operators and its relation with equivariant Maslov indices. Combining the notions in Sections \ref{sec4} and \ref{sec5}, we study an equivariant analog of Maslov triple indices which can be expressed as an algebraic combination of equivariant Maslov indices associated to three pairs of $H$-Lagrangians. Sections \ref{sec7} and \ref{sec8} are dedicated to  the study of equivariant $\zeta$-determinants of self-adjoint elliptic equivariant Fredholm operators and its fundamental role in the equivariant analog of Scott-Wojciechowski's theorem. Lastly, in Section \ref{sec9} we discuss the equivariant $\eta$-invariants of Dirac operators on manifolds with boundary and its relation with the equivariant Maslov triple index associated to a triple of boundary projections.   \\

\vspace{-1cm}
\section{Preliminaries} 
\label{sec2}

Let $M$ be a smooth compact odd-dimensional $\spinc$ manifold with boundary $\partial M.$ Assume that an identification between a neighborhood of $\partial M$ and $(-\epsilon,0] \times \partial M$ is fixed. Let $E$ be a complex vector bundle over $M,$ equipped with a Hermitian metric and a unitary connection $\nabla^E$. For $S^c(TM)$ the spin$^c$ spinor bundle over $M$, consider the tensor product bundle $S^c(TM) \otimes E$ equipped with a compatible connection. Assume that the spin$^c$ structure is fixed on $M.$ For simplicity, we denote the space of $L^2$-sections $L^2(S^c \otimes E)$ by $L^2(E)$ with the underlying bundle $S^c$ implicitly implied. Let $D^E : L^2(E) \to L^2(E)$ be the twisted Dirac operator on $M$ which takes the form
\begin{equation}\label{eq:dirac1}
	D^E= \gamma \Big( \frac{\partial}{\partial t}+ A\Big)
\end{equation}  near the cylinder $(-\epsilon,0] \times \partial M$ equipped with the product metric. Here, $\gamma: E_{|_{\partial M}} \to E_{|_{\partial M}}$ is a bundle isomorphism such that 
\begin{equation}\label{eq:gamma1}
	\gamma^2=-\id, \;\;\; \gamma^*=-\gamma
\end{equation}
and $A: L^2(E_{|_{\partial M}}) \to L^2(E_{|_{\partial M}})$ is the twisted Dirac operator on $\partial M.$ In particular, since $\partial M$ is closed, $A$ is essentially self-adjoint and has point spectrum. 
In view of \eqref{eq:dirac1} and \eqref{eq:gamma1}, $A$ anti-commutes with $\gamma,$
\begin{equation} \label{eq:gamma2}
	\gamma A =-A\gamma.
\end{equation} 
Since $M$ is odd-dimensional, $\gamma$ may be assumed to take the form $\diag(i, -i),$ see \cite[\S 9]{7}.
As a consequence, the space $L^2(E|_{\partial M})$ splits into the direct sum of 
\[ 
\cE_i:=L^2(E_i)=\ker(\gamma-i) \quad \text{ and } \quad 
\cE_{-i}:= L^2(E_{-i})=\ker(\gamma +i)
\]
whose grading is given by $\gamma.$ 


Let $\{\lambda,\psi_\lambda \}$ be the spectral decomposition of $A.$ Let $P^+: L^2(E_{|_{\partial M}}) \to L^2(E_{|_{\partial M}})$ be the Atiyah-Patodi-Singer (APS) spectral projection defined by  
\[
P^+(\sum_{\lambda}a_\lambda \psi_\lambda ) = \sum_{\lambda >0} a_\lambda \psi_\lambda.
\]  





Let $G$ be a compact group of isometries acting on $M$ preserving the fixed $\spinc$ structure. For $h \in G,$ we consider $H= \overline{\langle h \rangle}$ to be the closed subgroup of $G$ generated by $h.$ 

\begin{assumption} 
\label{assump1}
Let $h \in H.$ Assume that $h$ commutes with the Dirac operator $D.$
\end{assumption} 

In view of \eqref{eq:dirac1}, it follows that $h \in H$ commutes with the self-adjoint operator $A$ and the Clifford multiplication $\gamma.$ In particular, there is an inner-product-preserving induced action of $h \in H$ on the bundle $L^2(E_{|_{\partial M}}).$
Let $P: L^2(E_{|_{\partial M}}) \to L^2(E_{|_{\partial M}})$ be a projection. Then, $P$ is said to be \textit{$h$-equivariant} if $P$ commutes with $h\in H.$


\begin{definition} \label{laggrass0}
	Let $h \in H.$ The $h$-equivariant self-adjoint Fredholm Grassmannian of $A,$ denoted by $\Gr_h(A),$ is the set of all $h$-equivariant projections $P: L^2(E_{|_{\partial M}}) \to L^2(E_{|_{\partial M}})$ such that 
	\begin{enumerate}
		\item $P$ is a zeroth order pseudodifferential projection; 
		\item $P^2=P$ and $P^*=P,$ i.e. $P$ is an orthogonal projection;
		\item $\gamma P \gamma^* =1-P,$ i.e. $P$ is Lagrangian;
		\item $(P,P^+)$ is a Fredholm pair, i.e. $PP^+: \text{im}(P^+) \to \text{im}(P)$ is a Fredholm operator.
	\end{enumerate}
\end{definition}

This is an immediate equivariant adaptation of \cite[Definition 2.1]{15}. Note that the classical APS projection $P^+ \notin \Gr_h(A)$ if and only if $A$ is not invertible. In fact, $A$ is in general not invertible in \eqref{eq:dirac1}. Now, we modify $P^+$ by a finite rank projection $P_{\cL},$ i.e.
\begin{equation}
	\label{eq:perturbedAPS}
	P^{\partial}(\cL):=P^+ + P_{\cL}
\end{equation} 
where $\cL \subset \ker(A)$ is a Lagrangian subspace such that $\gamma(\cL)=\cL^\perp \cap \ker(A).$ Since $h$ commutes with $P_{\cL}$, its corresponding Lagrangian $\cL$ is an $h$-space for $h \in H.$ Then, one can verify that $P^{\partial}(\cL) \in \Gr_h(A),$ which shows that the set $\Gr_h(A)$ is indeed non-empty. The projection $P^{\partial}(\cL)$ is \textit{a priori} dependent on the choice of a Lagrangian subspace $\cL.$ Henceforth, we shall assume a choice of $\cL$ is fixed and write $P^{\partial}=P^{\partial}(\cL).$ 

Consider the subset   
\begin{equation}\label{eq:grinfty1}
	\Gr^{\infty}_h(A)=\{P \in \Gr_h(A) \;|\; P-P^+ \in \Psi^{-\infty}\}
\end{equation} 
where $\Psi^{-\infty}$ stands for the algebra of smoothing operators. Recall that there is a canonical projection in $\Gr(A)$ that depends only on $M$ and $D,$ which is commonly known as the Calder\'on projection $\cP_M$ on $M.$ It is defined as the orthogonal projection onto the Cauchy-data space 
\begin{equation} \label{eq:CDsp}
	\cL_M:=r\{\ker D : H_{1/2} (E) \to H_{-1/2}(E)\} \subset L^2(E|_{\partial M})
\end{equation}
where $r$ denotes the restriction to the boundary. See \cite{7}. By assumption, $h$ commutes with $\cP_M.$  This $h$-equivariant Calder\'on projection is an element of $\Gr_h(A).$ In fact, an equivariant refinement obtained from \cite{23} implies that $\cP_M \in \Gr^{\infty}_h(A).$

Let $P \in \Gr_h(A)$ be a self-adjoint spectral projection with respect to the boundary (tangential) operator. Then, the operator $D$ equipped with $P$ whose domain restricts to 
\[
\text{Dom}(D_P) = \{ s \in L^2(E) \;|\; s \in H^1(E), P(s_{|_{\partial M}})=0 \},
\]
is self-adjoint and Fredholm. 

As explained in \cite[\S 2]{15}, every orthogonal projection $P \in \Gr_h(A)$ takes the explicit form  
\begin{equation} \label{eq:projform1}
	P=\frac{1}{2} \begin{pmatrix} \id & T^* \\ T & \id \end{pmatrix}
\end{equation} for some odd unitary operator $T : \cE_i \to \cE_{-i}.$ On the other hand, given such an odd unitary operator $T,$ the matrix in $T$ on the right side of \eqref{eq:projform1} satisfies all four conditions in Definition~\ref{laggrass0}. We shall use the notation $T$ and $\Phi(P)$ interchangeably. The following lemma is a characterisation of $h.$


\begin{lemma} 
\label{hchar1}
	Let $h \in H.$ Let $P \in \Gr_h(A).$ By Assumption~\ref{assump1}, the isometry $h$ can be characterized by the form  
	\begin{equation} \label{eq:hmat1}
		h=
		\begin{pmatrix}
			a & 0 \\
			0 & \Phi(P)a\Phi(P)^*
		\end{pmatrix}.
	\end{equation}
\end{lemma}

\begin{proof}
	As $h$ acts as an isometry on $M,$ which lifts to act on $\Z_2$-graded
	Hilbert spaces $L^2(E)=L^2_+(E) \oplus L^2_-(E)$ and $L^2(E_{|_{\partial M}})=L^2_+(E_{|_{\partial M}}) \oplus L^2_-(E_{|_{\partial M}})= \cE_i \oplus \cE_{-i},$ whose gradings are given by $\gamma,$ we write $h$ in the form 
	\[
	h=\begin{pmatrix}
		a & b \\
		c & d
	\end{pmatrix} .
	\] 
	Let $T=\Phi(P).$ Note that 
	\[
	hP= \frac{1}{2}\begin{pmatrix}
		a & b \\
		c & d
	\end{pmatrix} 
	\begin{pmatrix}
		\id & T^* \\
		T & \id
	\end{pmatrix} =
	\begin{pmatrix}
		a+bT & b+aT^* \\
		c+dT & d+cT^*
	\end{pmatrix}
	\] and 
	\[
	Ph= \frac{1}{2}\begin{pmatrix}
		\id & T^* \\
		T & \id
	\end{pmatrix} 
	\begin{pmatrix}
		a & b \\
		c & d
	\end{pmatrix} =
	\begin{pmatrix}
		a+T^*c & b+T^*d \\
		c+Ta & d+Tb
	\end{pmatrix}
	\] 
	By the commutative relation $Ph =hP,$ we obtain
	\[
	h=\begin{pmatrix}
		a & T^*cT^* \\
		TbT & TaT^*
	\end{pmatrix}.  
	\] 
	On the other hand, by Assumption 1, $h$ must also commute with $\gamma = \diag(i,-i)$  which satisfies \eqref{eq:gamma1}. Then, we can verify by a direct computation that $h$ commutes with $\gamma$ if and only if 
	$TbT=0$ and  $T^*cT^*=0.$ Thus, the isometry $h$ is a diagonal matrix 
	\[
	h=\begin{pmatrix}
		a & 0 \\
		0 & TaT^*
	\end{pmatrix} \in \cU(\cE_i)\oplus \cU(\cE_{-i}).  
	\] 
\end{proof}

\begin{remark}
	In general,  the explicit form of the isometry $h$ does not depend on a particular projection. Hence, shall no confusion occur we will denote $h=\diag(a,WaW^*)$ for some $W \in \cU(\cE_i,\cE_{-i}).$  
\end{remark}

Let $(P,Q)$ be a Fredholm pair, i.e. $PQ: \im(Q) \to \im(P)$ is a Fredholm operator. By Assumption ~\ref{assump1}, the pair is compatible with the $h$-action, i.e. $(PQ)h=h(PQ)$ for all $h \in H.$ In view of \eqref{eq:projform1}, the product of the corresponding matrices is well-defined. In this case, for $i=1,2,3,$ let $T_i$ be the corresponding odd unitary operators  in defining $P,Q$ and $h$ respectively, then $T^*_iT_j$ must commute with $a$ for all $i,j=1,2,3$ with $i\neq j$ and $T_3^*(T_1T^*_2)T_3$ must also commute with $a.$   The explicit formula for the general matrix multiplication $P_\gamma \cdots P_\beta P_\alpha$ is significantly more convoluted to state so we leave this to interested readers. 

\begin{lemma}\label{PQlemma1}
	Let $h=\diag(a,WaW^*) \in H.$ Let $P,Q \in \Gr_h(A)$ be given by \eqref{eq:projform1}, i.e. 
	$$P=\frac{1}{2} \begin{pmatrix} \id & T^* \\ T & \id \end{pmatrix}, \quad Q=\frac{1}{2} \begin{pmatrix} \id & S^* \\ S & \id \end{pmatrix}$$ for some odd unitary operators $T,S,W : \cE_{i} \to \cE_{-i}.$ Then,
	\begin{center}
		\begin{enumerate}
			\item $(P,Q)$ is a Fredholm pair if and only if $-1 \notin \spec_{ess}((a-\id)+T^*S)),$
			\item $(P,Q)$ is an invertible pair (i.e. $PQ: \im(Q) \to \im(P)$ is invertible) if and only if $-1 \notin \spec((a-\id)+aT^*S),$
			\item $\ker(P) \cap \im(Q) \cong  \ker (a(\id + T^*S)),$ 
			\item $h(P-Q) \in \Psi^{-\infty}$ if and only if $(\id -T^*S)a \in \Psi^{-\infty}.$ 
		\end{enumerate} 
	\end{center}
\end{lemma}
\begin{proof}
	We follow an argument of \cite{15}. Let $L^2(E|_{\partial M})=\cE_i \oplus \cE_{-i}.$ Then, any element in $L^2(E|_{\partial M})$ can be written as $(x, Sy)^T$ for $x,y \in \cE_i$ and for some $S:\cE_i \to \cE_{-i}.$ One verifies that $\im(Q)$ is the set $\{(x, Sx)^T\;|\; x\in \cE_{i}\}.$ Then, the composition $h(PQ)$ takes the form 
	\begin{equation}\label{eq:hPQ1}
		hPQ=\frac{1}{2} \begin{pmatrix} a(\id + T^*S) \\ Ta(\id+ T^*S)\end{pmatrix}
	\end{equation} 	
	where we use the fact that both $T^*W$ and its adjoint $W^*T$ commute with $a.$ Hence, $(P,Q)$ is a Fredholm pair if and only if $a(\id + T^*S)=\id + [(a-\id) + aT^*S]$ is Fredholm, which is exactly when $-1 \notin \spec_{ess}((a-\id)+aT^*S).$  Similarly, $(P,Q)$ is an invertible pair if and only if the condition $-1 \notin \spec((a-\id)+aT^*S)$ holds. For (3), in view of  \eqref{eq:hPQ1} the intersection $\ker(P) \cap \im(Q)$ coincides with $\ker(a(\id + T^*S)).$ For the last part, one verifies that 
	\begin{equation}
		h(P-Q)=\frac{1}{2} \begin{pmatrix} 0 & a(T^* - S^*) \\ (T-S)a & 0\end{pmatrix}.
	\end{equation} Since $T,S$ and $a$ are unitary, the operator $(T-S)a$ is smoothing if and only if $(\id - T^*S)a$ is smoothing. 
\end{proof}

Let $\cH$ be a complex separable Hilbert space with a fixed basis. Let $H$ be the compact group of isometries on $M.$ Let 
\begin{equation} 
	\label{eq:Hprime}
	\cH'=L^2(H,\cH)
\end{equation} be the space of square integrable functions on $H$ with values in $\cH.$ These $L^2$ functions are integrated with respect to an $H$-invariant Haar measure.  By the Peter-Weyl theorem, the $H$-module $\cH'$ can be expressed as an isomorphism 
\begin{equation}
	L^2(H,\cH) \cong  \sum_i (\dim V_i) V_i \otimes \cH
\end{equation} where the sum runs through irreducible complex representations $V_i$ of $H$ for all $i.$ Let $\cB(\cH'), \cF(\cH'), GL(\cH')$ and  $\cU(\cH') $ be the Banach algebra of bounded operators, the subalgebra of Fredholm operators of the form $I+\cK,$ the general linear group and the unitary group of $\cH'$ respectively.  The natural $H$-action on $\cH'$ induces one on $\cB(\cH'),$ which is given by the commutative diagram

\begin{center}
	\begin{tikzcd}
		\cH' \arrow[r,"h"] \arrow[d,"T"] & \cH' \arrow[d,"T^h"] \\
		\cH' \arrow[r,"h"] & \cH',
	\end{tikzcd}
\end{center}
i.e. 
\begin{equation}
	T^h(f)=hT(h^{-1}f) 
\end{equation} for $f \in \cH'$ and $T,T^h \in \cB(\cH').$ Such an $H$-action respects the algebra structure and is continuous with respect to the norm topology of $\cB(\cH').$ The induced actions on the spaces $\cF(\cH'), GL(\cH')$ and $\cU(\cH')$ are defined similarly. 



Before proceeding to the next section, we give the following definition. Recall that a function $f$ is called $H$-equivariant if and only if $fh(x)=hf(x)$ for all $x\in X.$

\begin{definition}
Let $X$ be an $H$-equivariant CW-complex. Let $x_0 \in X.$ Define the \textit{equivariant fundamental group} of $X$ based at $x_0$ to be the set of  homotopy classes of $H$-equivariant continuous loops at $x_0,$ where the equivalence is by \textit{equivariant homotopy}, that is, 
\[
\pi_1(X,x_0)_H = \left\{f:[0,1] \to X 
 \mid f \text{ is } H\text{-equivariant, continuous}, f(0)=f(1)=x_0\right\}/ \sim
\]
where $\sim$ is the  relation given by: let $f,g : [0,1] \to X$ be $H$-equivariant continuous functions such that $f(0)=f(1)=x_0$ and $g(0)=g(1)=x_0.$ Then, $f\sim g$ if and only if there exists an $H$-equivariant continuous function $F:[0,1] \times [0,1] \to X$ such that $F(0,s)=F(1,s)=x_0,$  $F(t,0)=f(t),$ and $F(t,1)=g(t).$
\end{definition}

We remark that if $H$ acts trivially on a space $Y,$ then $h(y)=y$ for $y \in Y.$ Thus, in this case $f$ being $H$-equivariant on $Y$ is equivalent to $f(y)=h(f(y))$ for all $y\in Y,$ and so $f(y)$ lies in the fixed point set $Y^H.$ In particular, in the equivariant homotopy relation above, $H$ is assumed to act trivially on $Y=[0,1].$


\section{Equivariant spectral flow} \label{sec3}


Let $\widehat{\cF}(\cH')$ be the subalgebra of $\cF(\cH')$ of all self-adjoint equivariant Fredholm operators. By \cite{5}, the space $\widehat{\cF}(\cH')$ consists of three connected components: 
\begin{equation}
	\widehat{\cF}(\cH')=\widehat{\cF}_+(\cH')\oplus \widehat{\cF}_-(\cH') \oplus \widehat{\cF}_*(\cH').
\end{equation} Here, $\widehat{\cF}_{+}(\cH')$ (resp. $\widehat{\cF}_{-}(\cH')$) consists of those operators whose essential spectrum lie \textit{entirely} in $\R^{+}$ (resp. $\R^{-}$); whereas $\widehat{\cF}_*(\cH')$ consists of those operators whose essential spectrum lie \textit{simultaneously} in both $\R^{+}$ and $\R^-.$ 

In fact, for an element $T\in \widehat{\cF}_{+}(\cH')$ and for all $0 \leq s \leq 1,$ the linear homotopy 
\[
T(s)=(1-s)T + s\id
\] contracts $T$ to the identity operator $\id.$ Similarly, the homotopy $T(s)=(1-s)T + s(-\id)$ contracts $T \in \widehat{\cF}_{-}(\cH')$ to $-\id.$ This concludes the following statement.

\begin{lemma}
	Both of the spaces $\widehat{\cF}_{\pm}(\cH')$ are contractible. 
\end{lemma}

Let $\pi : \cB(\cH') \to \cQ(\cH')=\cB(\cH')/\cK(\cH')$ be the canonical projection into the Calkin algebra of $\cH'.$ Let $\cJ$ be the group of invertible elements of $\cQ(\cH').$  By definition, we have $\cF=\pi^{-1}(\cJ).$ Let $\widehat{\cJ} \subset \cJ$ be the subspace that contains all self-adjoint invertible elements of $\cQ(\cH').$ Then, by \cite[Lemma 16.3]{7} the projection $\pi$ restricts to a homotopy equivalence $\widehat{\cF}_*(\cH') \xrightarrow{\sim} \widehat{\cJ}_*(\cH')$ where $\widehat{\cJ}_*(\cH')$ is the non-trivial connected component of $\widehat{\cJ}.$ More precisely, 
\[
\widehat{\cJ}_*(\cH') = \{j \in \cQ(\cH') | j^*=j, \text{spec}(j) \cap \R^+ \neq \emptyset \textit{ and } \text{spec}(j) \cap \R^- \neq \emptyset\}
\] 
%
where the pre-image of each $j \in \widehat{\cJ}_*$ takes the form $\frac{1}{2}(T+T^*) \in \widehat{\cF}_*$ for which $\pi(T)=\pi(T^*)=j.$ The induced $H$-action on $\widehat{\cJ}_*$ is given by that  restricted on $\widehat{\cF}_*$ and by the equivariance of projections. 

Moreover, as with the case where invertible elements can be retracted into unitary elements, the standard $C^*$-algebra retraction 
\begin{equation}\label{eq:retract1}
	j_s= j \Big((1-s)\id+\frac{s}{\sqrt{j^*j}}\Big),  \;\;\; \text{for } s \in [0,1],
\end{equation}
defines the unitary groups 
\begin{align*}
	\widehat{J}(\cH')=\{j \in \cQ(\cH')\;|\; j^*=j, j^2=\id, j \neq \pm \id \}, \\
	\widehat{J}_*(\cH')=\{j \in \widehat{J}(\cH')\;|\; \spec(j)=\{-1,+1\}\}.
\end{align*}
By the unitary retraction \eqref{eq:retract1}, we conclude that $\widehat{\cF}_*(\cH')$ is homotopy equivalent to $\widehat{J}_*(\cH').$ See \cite[Theorem 16.4]{7}.

For $h \in H,$ let $P^+$ be the $h$-equivariant positive spectral projections of $D$  onto the Lagrangian subspace $\cL_{P^+}(\cH').$ Let $P^-$ be the complementary projection of $P^+$ onto the Lagrangian subspace $\cL_{P^-}(\cH').$ Then,  $P^+-P^-$ defines an involution in $\widehat{J}_*(\cH').$ Choose and fix a finite dimensional projection $P_k$ of $P^-$ onto an $h$-subspace $V_k$ of dimension $k$ of Ran($P^-$). Then, in view of $\widehat{\cF}_*(\cH') \simeq \widehat{J}_*(\cH'),$ every equivariant Fredholm operator $T \in \widehat{\cF}_*(\cH')$ corresponds to some involution up to finite rank projection $P^+ -P^- + P_k.$ Then, any element of $\pi_1(\widehat{\cF}_*(\cH'))_H$ is a family $\{B_k(t):=P^+-P^-+2tP_k\}_{t \in S^1}.$ For simplicity, we shall denote a general element of $\pi_1(\widehat{\cF}_*(\cH'))_H$ by $B_t,$ to mean a family of  self-adjoint equivariant Fredholm operators labelled by $t \in S^1.$

\begin{assumption}
	By choosing a good partition of a grid, we mean to choose a partition $0 < t_0 < \cdots <t_n=1$ of the `$x$-axis' the interval $[0,1],$ and a partition of the `$y$-axis', usually a set of real number $\{a_1,a_2,\ldots,a_n\}$ which is the image of some continuous function, such that the function is non-singular at the four corner vertices. This can always be achieved by some $\epsilon$-perturbations. A refinement of a chosen partition refers to the $1/2^i$-division of the partition into smaller grids such that it sums to the original grid.  
\end{assumption}

\begin{definition} \label{equivspecflowdef}
	Let $h \in H.$  By Assumption 2, fix a good partition of a grid where the $x$-axis is the interval $[0,1]$ and the $y$-axis is a set $\{a_j\}^n_{j=1}$ of the image of spectral projections. For each $j=1,2,\cdots,n,$ let $P^+_j$ be the positive spectral projection of $B_t \in \pi_1(\widehat{\cF}_*(\cH'))_H$ onto the $y$-axis over the partitioned grid with the subinterval $[t_{j-1},t_j]$ of $[0,1].$ Then, we define the $h$-equivariant spectral flow of $B_t$ by
	\begin{align} \label{eq:hspec1}
		\spf_h : \pi_1&(\widehat{\cF}_*(\cH'))_H \longrightarrow \C \nonumber \\
		\spf_h (B_t) = \sum^n_{j=0} &\Big[\tr\big(h|_{\E_j(t_j)}\big) - \tr\big(h|_{\E_j(t_{j-1})}\big)\Big]
	\end{align}
	where $\E_j(t_i)$ denotes the corresponding $P^+_j(t_i)$-eigenspaces for $i=j-1,j.$
\end{definition}

\begin{proposition} \label{homotopyinv}
For $h \in H,$ the $h$-equivariant spectral flow is a homotopy invariant. That is, for a homotopy $H(s,t): [0,1] \to \pi_1(\widehat{\cF}_*(\cH'))_H$ between $B(t)$ and $B'(t)$ such that $B(0)=B'(0)$ and $B(1)=B'(1),$ we have
	\[
	\spf_h(B(t)) = \spf_h(B'(t)). 
	\]
\end{proposition}

\begin{proof}
	The proof follows verbatim from \cite[Proposition 3]{22}. In particular, the main strategy is to break a homotopy $H: I \times I \to \widehat{\cF}_*(\cH')$ into the sum of `shorter' homotopies in each grid. This is achievable due to compactness and we can cover $I \times I$ by a finite set of preimages $\{H^{-1}(N_i)\}_i$ where each $N_i$ is a neighbourhood of $B(t)$ which covers some grid $J_i\times J_j.$  
\end{proof}

Proposition~\ref{homotopyinv} allows us to consider continuous deformations of spectrum with fixed points. In particular, given a continuous path which crosses the $x$-axis over a region in a partitioned grid, as in Figure 1, we can consider a homotopy of paths with fixed endpoints, as in Figure 2. By homotopy invariance, the equivariant spectral flows of the green line and the blue line are equal, i.e. the net `counting' is the same as that of in Figure 3. The local equivariant spectral flow is thus given by $\tr(h|_{\E(t_j)}) - \tr(h|_{\E(t_{j-1})}).$

\begin{figure}[h]
	\centering
	\begin{minipage}{.32\textwidth}
		\centering	
        \includegraphics[width=.9\linewidth]{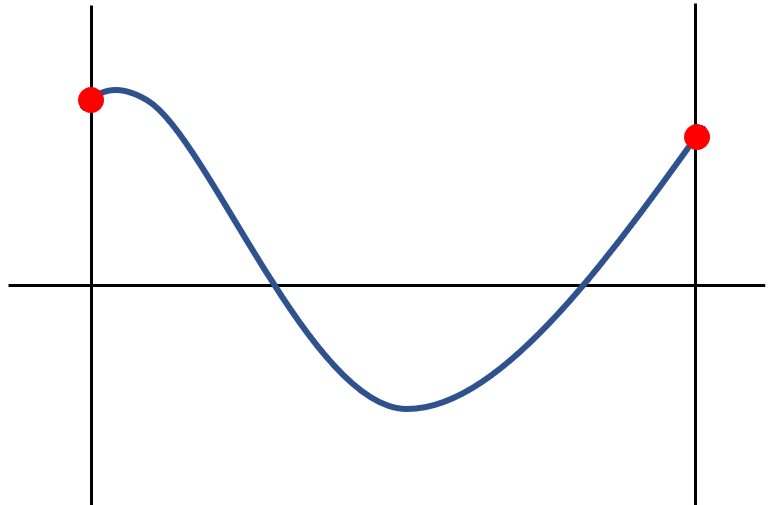}
		\captionof{figure}{\footnotesize A path with crossings, and endpoints in the same quadrant.}
	\end{minipage} 
	\begin{minipage}{.32\textwidth}
		\centering
		\includegraphics[width=.9\linewidth]{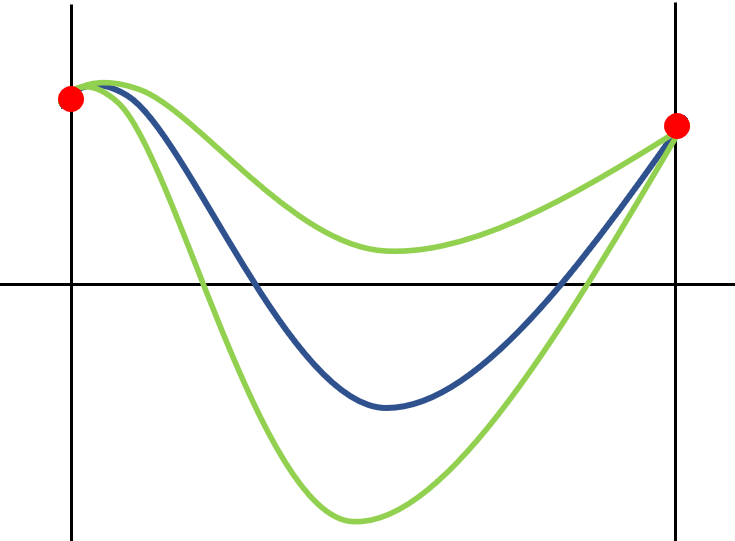}
		\captionof{figure}{\footnotesize A path homotopy relative endpoints in the same quadrant.}
	\end{minipage}
	\begin{minipage}{.32\textwidth}
		\centering
		\includegraphics[width=.9\linewidth]{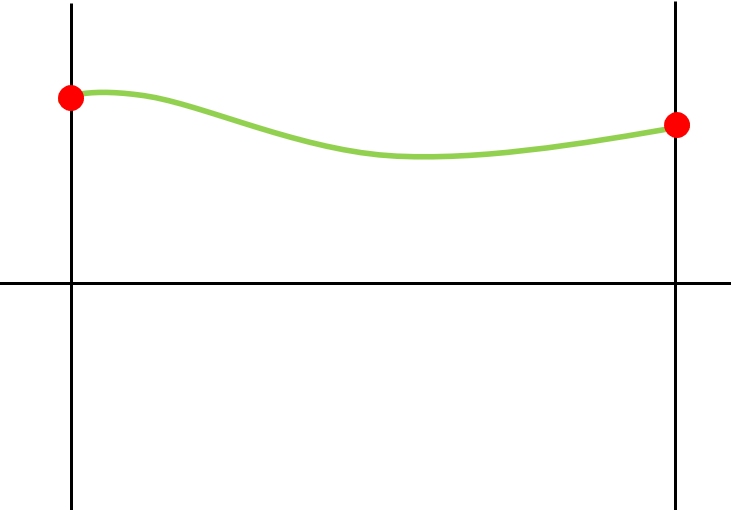}
		\caption{\footnotesize The resulting path without crossings.}
	\end{minipage}
\end{figure} 

The elegant method of Phillips \cite{22} simplifies the notion of spectral flow introduced by Atiyah-Patodi-Singer \cite{3} which measures the net change of crossings across zero counted with multiplicities. In particular, Phillips's counting method no longer depends on the crossings over the $x$-axis because via homotopy we only need to consider the crossings on the \textit{positive} `quadrant' with respect to the $y$-axis. In the equivariant case, this corresponds to taking the trace of the induced linear map on the $\E(t_{j-1})$ and $\E(t_j)$ respectively.

In general, there are much more complicated paths of spectrum. However, by Phillips's method, over an appropriate partition, we may reduce this to only three simple types of path crossings and they span all other possible combinations. The first type has already been described above. The second type are shown in Figure 4-6. They depict the net effect of a crossing from the negative to the positive `quadrant'. The local equivariant spectral flow is $\tr(h|_{\E(t_j)}).$ Similarly, Figure 7-9 depicts the net effect of a crossing from the positive to the negative `quadrant', for which the local equivariant spectral flow is $-\tr(h|_{\E(t_{j-1})}).$ 
\begin{figure}[h]
	\centering
	\begin{minipage}{.32\textwidth}
		\centering
		\includegraphics[width=.9\linewidth]{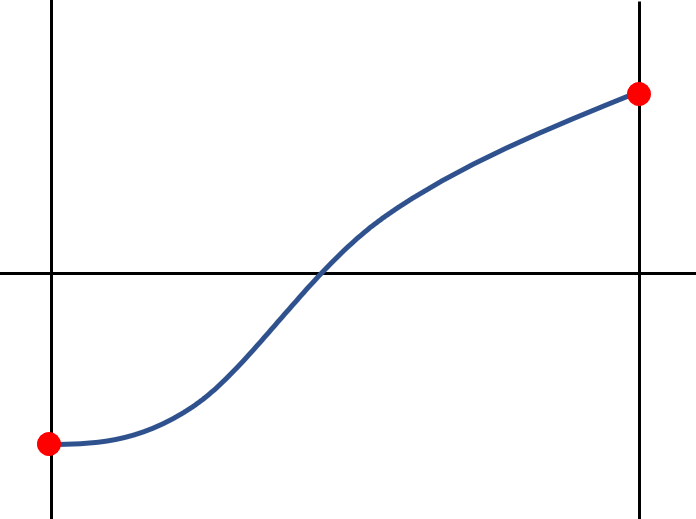}
		\captionof{figure}{\footnotesize A path with crossings, and endpoints in different quadrants.}
	\end{minipage}
	\begin{minipage}{.32\textwidth}
		\centering
		\includegraphics[width=.9\linewidth]{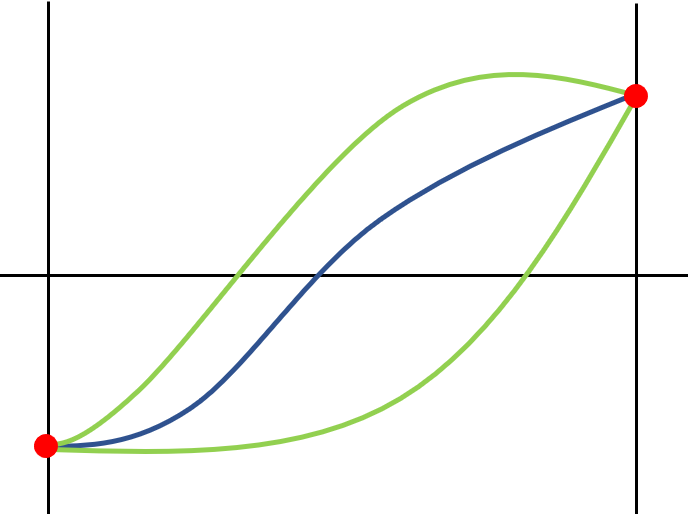}
		\captionof{figure}{\footnotesize A path homotopy relative endpoints in different quadrants.}
	\end{minipage}
	\begin{minipage}{.32\textwidth}
		\centering
		\includegraphics[width=.9\linewidth]{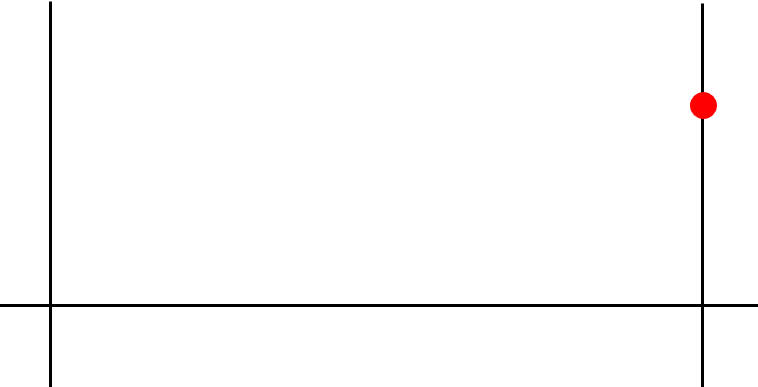}
		\captionof{figure}{\footnotesize Net effect of crossings from negative to positive quadrants.}
	\end{minipage}
\end{figure} 
\begin{figure}[h]
	\centering
	\begin{minipage}{.32\textwidth}
		\centering
		\includegraphics[width=.9\linewidth]{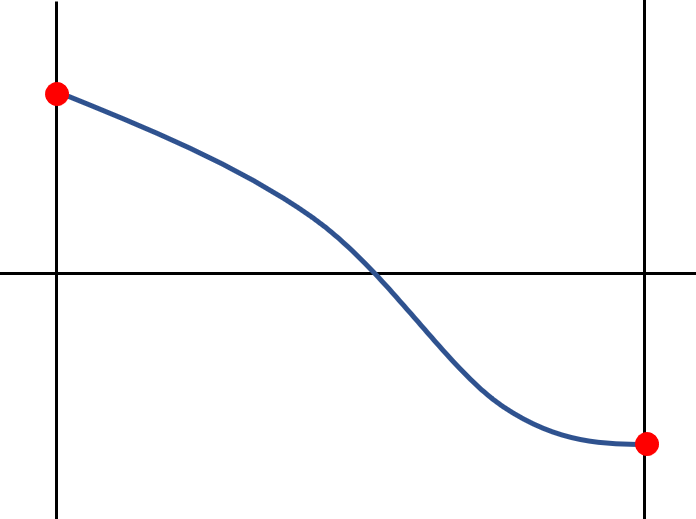}
		\captionof{figure}{\footnotesize Another case of crossings with endpoints in different quadrants.}
	\end{minipage}
	\begin{minipage}{.32\textwidth}
		\centering
		\includegraphics[width=.9\linewidth]{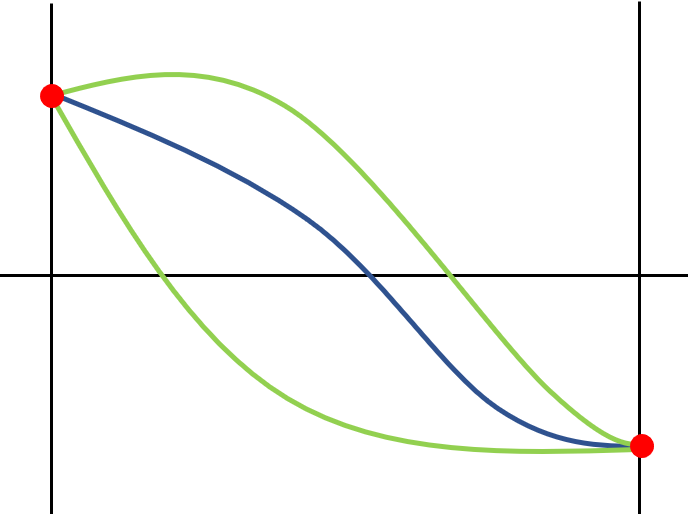}
		\captionof{figure}{\footnotesize A path homotopy relative endpoints in different quadrants.}
	\end{minipage}
	\begin{minipage}{.32\textwidth}
		\centering
		\includegraphics[width=.9\linewidth]{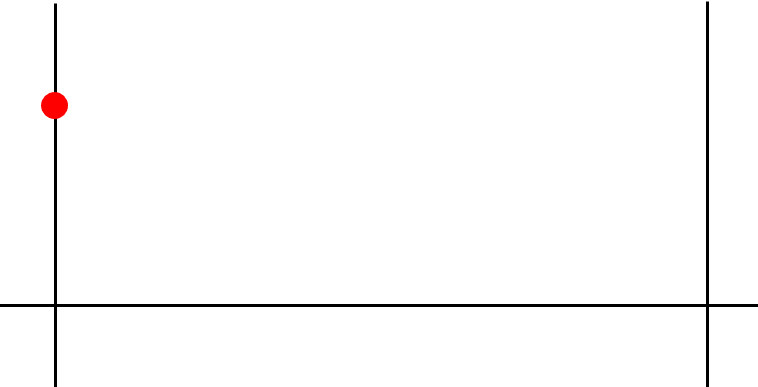}
		\captionof{figure}{\footnotesize Net effect of crossings from positive to negative quadrants.}
	\end{minipage}
\end{figure} 

\begin{figure}[h]
	\centering
	\begin{minipage}{.32\textwidth}
		\centering
		\includegraphics[width=.9\linewidth]{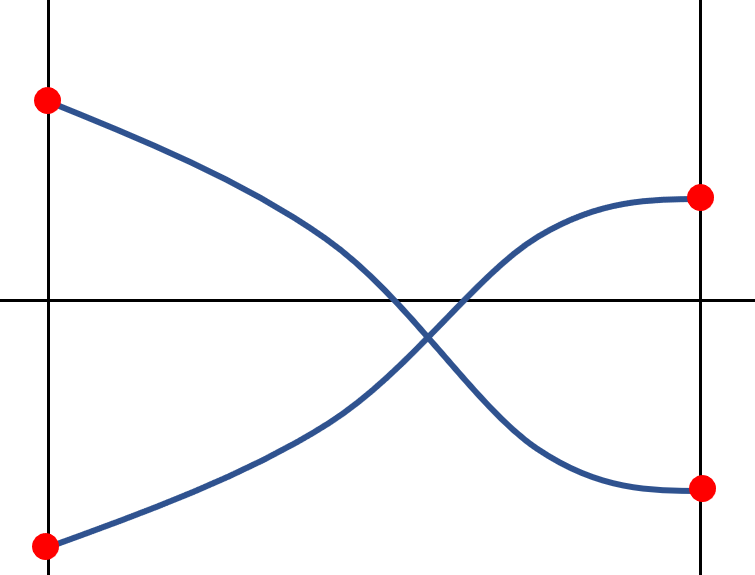}
		\captionof{figure}{\footnotesize A mixed case of crossings across different quadrants.}
	\end{minipage}
	\begin{minipage}{.32\textwidth}
		\centering
		\includegraphics[width=.9\linewidth]{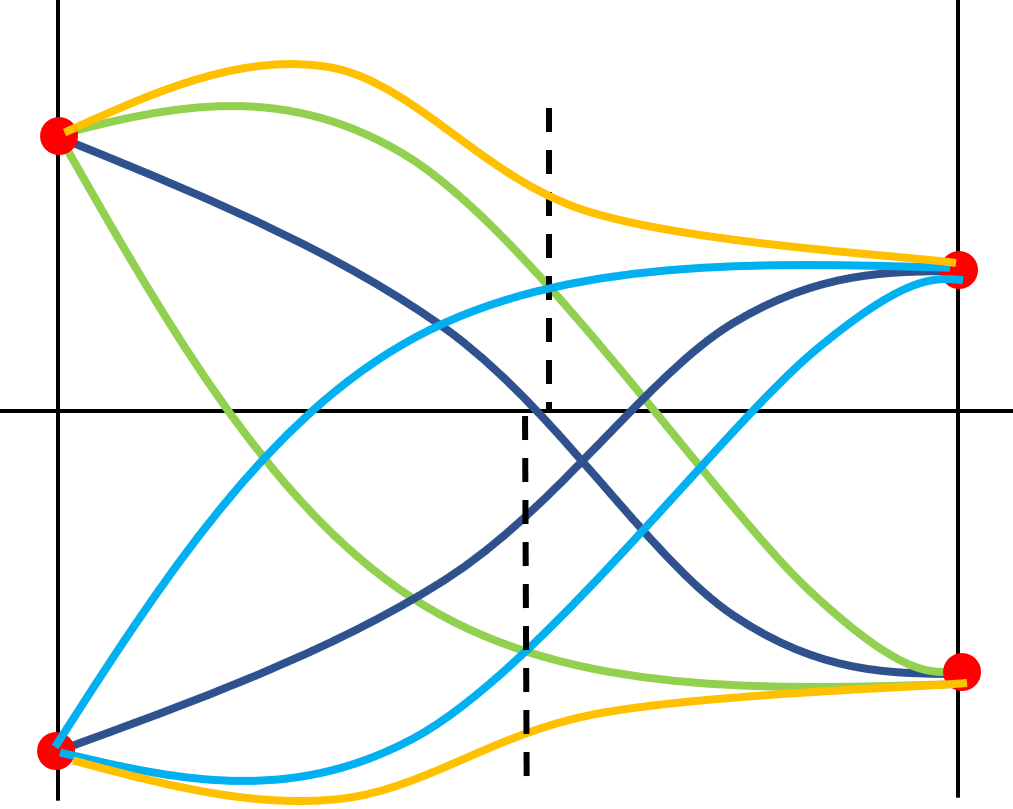}
		\captionof{figure}{\footnotesize Path homotopies of the mixed case.}
	\end{minipage}
	\begin{minipage}{.32\textwidth}
		\centering
		\includegraphics[width=.9\linewidth]{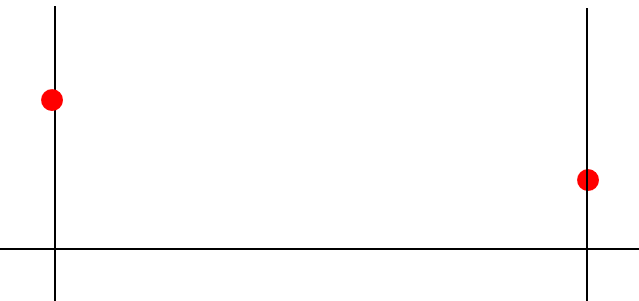}
		\captionof{figure}{\footnotesize Net effect of crossings in the mixed case.}
	\end{minipage}
\end{figure} 
The third type is the combination of crossed paths of the first and the second types, as shown in Figure 10-12. In the positive `quadrant' of Figure 11, we construct a new homotopy partly from the green line and the light blue line respectively and further deform smoothly to the orange line. We perform similar procedure for the bottom `quadrant'. The net effect is exactly as in Figure 12 where the local equivariant spectral flow is $\tr(h|_{\E(t_j)}) - \tr(h|_{\E(t_{j-1})}).$

\begin{proposition} \label{spfwelld}
	For $h \in H,$ the equivariant spectral flow $\spf_h$ is well-defined, i.e. it is independent of the refinement of the grid partition and it only depends on the continuous function $B(t): [0,1] \to \widehat{\cF}_*(\cH').$
\end{proposition}
\begin{proof}
	Without loss of generality, we choose a good grid partition such that there are only one to two eigenvalues inside the grid over $[t_{j-1},t_j].$ Let $t_d$ be any element in $[t_{j-1},t_j]$ (its corresponding line of spectrum is the vertical orange line, Figure 13). Let $P^+_d(t_d)\Big|^a_0$ be the intersection (green dot, Figure 13)  between the line $t=t_d$ and a path of spectrum. Then,  
	\begin{align*}
		\spf_h(B_t)_{\substack{t_{j-1} \leq t \leq t_j \\ 0 \leq P^+ \leq a}}
		&= \Big(\tr\big(h|_{\E_j(t_j)}\big) - \tr\big(h|_{\E_j(t_d)}\big) \Big)+ \Big(\tr\big(h|_{\E_d(t_d)}\big) - \tr\big(h|_{\E_d(t_{j-1})}\big) \Big) \\
		&= \tr\big(h|_{\E_j(t_j)}\big) - \tr\big(h|_{\E_j(t_{j-1})}\big)
	\end{align*} where the second equality follows from the fact that $P^+$ is canonical, that is, $P^+_j=P^+_d.$  This shows that $\spf_h$ is independent of vertical cut, see Figure 13. On the other hand, let the orange line be any horizontal cut as shown in Figure 14 which splits the grid $[0,a]$ into $[0,a']$ and $[a',a]$. Then, it is evident from definition that 
	$$\spf_h(B_t)_{\substack{t_{j-1} \leq t \leq t_j \\ 0 \leq P^+ \leq a}}
	= \spf_h(B_t)_{\substack{t_{j-1} \leq t \leq t_j \\ 0 \leq P^+ \leq a'}} +
	\spf_h(B_t)_{\substack{t_{j-1} \leq t \leq t_j \\ a' \leq P^+ \leq a}}.$$
	Thus, it is independent of horizontal cut. The general case follows immediately by summing over all grids.
	\begin{figure}[h]
		\centering
		\begin{minipage}{.4\textwidth}
			\centering
			\includegraphics[width=.9\linewidth]{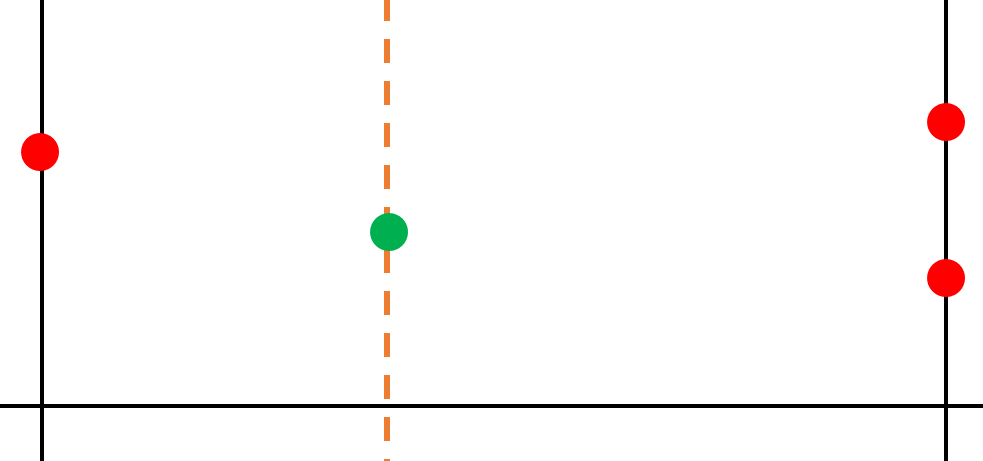}
			\captionof{figure}{Vertical cut.}
		\end{minipage}
		\begin{minipage}{.4\textwidth}
			\centering
			\includegraphics[width=.9\linewidth]{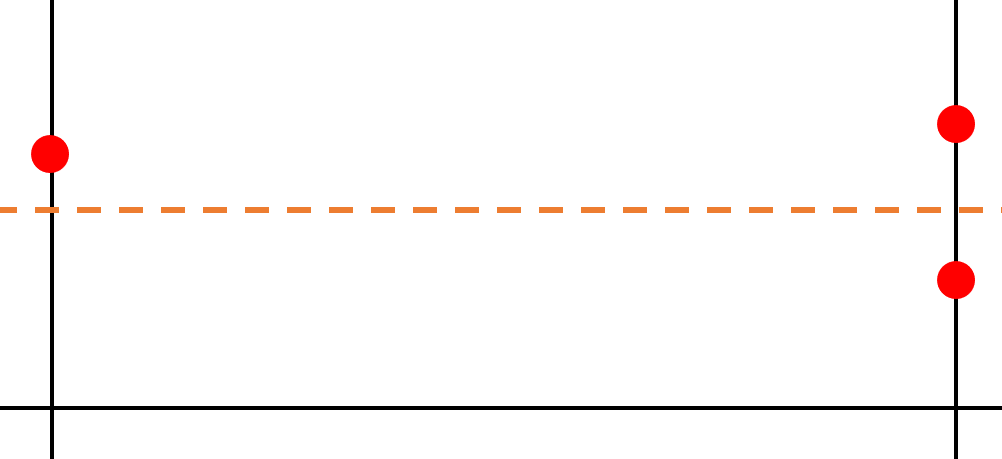}
			\captionof{figure}{Horizontal cut.}
		\end{minipage}
	\end{figure}
\end{proof}


Let $\spf_H : \pi_1(\widehat{\cF}_*(\cH'))_H \longrightarrow R(H)$ be the equivariant spectral flow such that the evaluation at every $h \in H$ is given by \eqref{eq:hspec1}. Here, $R(H)$ denotes the representation ring of $H,$ regarded as the ring of characters of $H.$

\begin{theorem} \label{specisom}
	The spectral flow $\spf_H : \pi_1(\widehat{\cF}_*(\cH'))_H \longrightarrow R(H)$ is an isomorphism.
\end{theorem}

\begin{proof}
	Inspired by \cite{7}, this is equivalent to show that the following diagram
	$$
	\begin{tikzcd}[column sep=small]
		\centering
		\pi_1(\widehat{\cF}_*(\cH'))_H \arrow[rr,"\ind_a"]\arrow[dr,"\spf_H"] &  & K^1_H(S^1)\arrow[dl,"ch_H",swap] \\ 
		& R(H) &
	\end{tikzcd}
	$$ commutes. Here, the equivariant $K^1$-theory of $S^1$ is isomorphic to the homotopy classes of equivariant maps $[S^1,U(\infty)]_H,$ 
    which is $\pi_1(U(\infty))_H$ by Bott's Periodicity.  Hence, $ch_H : K^1_H(S^1) \to \pi_1(U(\infty))_H \cong R(H)$ is an isomorphism. On the other hand, the top horizontal map is given by 
	\begin{align*}
		\ind_a: \pi_1(\widehat{\cF}_*(\cH'))_H &\longrightarrow K^1_H(S^1) \\
		B_k(t) &\longmapsto [\cW_t]
	\end{align*} where $\cW_t= -\exp(i \pi B_k(t))$ is a unitary map which equals $\exp(2 \pi i t)$ on $V_k$ and $\id$ on $V_k^\perp.$ Here, $V_k$ is the image of the infinite `negative half' projection  $P_{-}(\cH')$ with $P_++P_-=\id.$ In particular, the exponential map $B_k(t) \mapsto \exp(i\pi B_k(t)+\id)$ can be shown to be a homotopy equivalence between $\widehat{\cF}_\ast$ and $ U(\infty),$ see \cite{22,5,7}. Thus, their first equivariant fundamental groups are isomorphic given by the induced map $\ind_a.$ Moreover, by a direct computation
	\begin{align} \label{eq:chernind}
		ch_H(\ind_a(B_k(t)))
		&=\frac{1}{2\pi i} \int_0^1\tr \Big(h\cW^{-1}(t)\frac{d\cW(t)}{dt} \Big) dt  \\
		&=\frac{1}{2\pi i} \int_0^1 \tr(h|_{V_k}) e^{-2\pi i t}(2\pi i)e^{2\pi i t}dt \nonumber \\
		&= \tr(h|_{V_k}) \int_0^1 dt= \spf_h(B_k(t)) \nonumber
	\end{align} for all $h \in H.$ In particular, one deduces that the family ${B_1(t)}$ is mapped to a generator $\tr(h|_{V_1})$ in $R(H)$ via $\spf_h.$ Equation \eqref{eq:chernind} implies that $\spf_h=ch_H \circ \ind_a.$ Since both $ch_H$ and $\ind_a$ are isomorphisms, so is $\spf_H.$ 
\end{proof}

\begin{proposition}
	For $h\in H,$  $\spf_h$ satisfies the following properties.
	\begin{enumerate}
		\item For $t \in [0,1],$ let $B_k(t)$ and $B'_k(t)$ be in  $\pi_1(\widehat{\cF}_*(\cH'))_H$ such that $B_k(1)=B'_k(0).$ Denote the concatenation by $B_k(t)* B'_k(t).$ Then, $$\spf_h(B_k(t)* B'_k(t))=\spf_h(B_k(t)) + \spf_h(B'_k(t)).$$  
		\item The map $\spf_h$ is natural under any $h$-equivariant isomorphism on $\widehat{\cF}_*(\cH').$  
		\item If $h$ acts trivially, then the $h$-equivariant spectral flow coincides with the classical spectral flow of Phillips, i.e. $\spf_h=\spf.$  Furthermore, if $h$ acts trivially and $D(t)$ is a path of self-adjoint invertible Fredholm operators, then $\spf_h(D(t))=0.$
	\end{enumerate}
\end{proposition}


\section{Equivariant Maslov index} \label{sec4}
In this section, we discuss an equivariant analog of Maslov indices of Lagrangian paths  \textit{\`{a} la} Kirk-Lesch \cite{15}. Let $\cH'$ be the model $H$-module Hilbert space as above and let  $\gamma: \cH' \to \cH' $ be a unitary map satisfying \eqref{eq:gamma1}. If we consider a symplectic form on $\cH'$ defined by $\omega(x,y)= \langle x, \gamma y \rangle,$ then the triple $(\cH',\langle,\rangle,\gamma)$ forms a Hermitian symplectic Hilbert space.

\begin{definition}
	Given a Hermitian symplectic Hilbert space $(\cH',\langle,\rangle,\gamma),$ define the set of Lagrangian subspaces of $\cH'$ by 
	\[
	\fL(\cH')= \{\cL \subset \cH' \;|\; \gamma(\cL)= \cL^\perp \cap \cH'\}.
	\]
\end{definition}

\begin{remark}
	The image of any orthogonal projection $P$ on $\cH'$ satisfying the condition $\gamma P\gamma^*=\id -P$ is a Lagrangian subspace. Hence, $\cL(\cH')$ is non-empty. Moreover, $\Gr_h(\cH')$ can be identified  with  $\cL(\cH')$ by  
	\begin{equation} \label{eq:laggrassisom1}
		P \mapsto \im(P), \quad\quad  P_\cL \mapsfrom \cL.
	\end{equation} The $h$-action on each $\cL$ is induced by the restriction of $\cH'.$ This means that each Lagrangian $\cL$ is naturally an $h$-subspace. 
\end{remark}

\begin{definition} \label{laggrass1}
	Let  $\fL^2(\cH')=\{(\cL_1,\cL_2)\;|\; \cL_1,\cL_2 \in \fL(\cH')\}$ be the set of pairs of Lagrangian subspaces of $\cH'.$ Then, the subset of Fredholm pairs and invertible pairs are respectively  given by
	\begin{align} 
		\fL^2_F(\cH')&=\{(\cL_1,\cL_2) \in \fL^2(\cH') \;|\; (\cL_1,\cL_2) \text{ is a Fredholm pair} \} \label{eq:fredlagpair1} \\
		\fL^2_*(\cH')&=\{(\cL_1,\cL_2) \in \fL^2(\cH') \;|\; (\cL_1,\cL_2) \text{ is an invertible pair} \label{eq:invlagpair1} \}.
	\end{align}
	
	Let $\cL_M$ be the canonical Lagrangian subspace associated to the canonical Calderon projection $\cP_M$ of $M.$ Then, write $\fL^0_F(\cH')$ (resp. $\fL^0_*(\cH')$ ) to be the Lagrangian subset with respect to $\cL_M,$ i.e. all Lagrangian pairs are taking the form $(\cL,\cL_M)$ such that $(\cL,\cL_M)$ is a Fredholm (resp. invertible) pair. Then, a Maslov cycle $\cM_{\cL_M}$ at $\cL_M$ is defined by 
	\begin{equation} \label{eq:maslovcycle1}
		\cM_{\cL_M}= \fL^0_F(\cH') \backslash \fL^0_*(\cH').
	\end{equation} 
\end{definition}

\begin{remark}
	Recall that $(\cL_1,\cL_2)$ is a Fredholm pair for $\cL_i \in \fL(\cH')$ if $\cL_1 \cap \cL_2$ is a finite dimensional $h$-subspace of $\cH'$ and $\cL_1 + \cL_2$ is a closed $h$-space with finite codimension. In particular, the Lagrangian pair $(\cL_1,\cL_2)$ is invertible if $\cL_1 \cap \cL_2= \{0\}$  and $\cL_1 + \cL_2 = \cH'.$ Hence, $\cM_{\cL_M}$ is non-trivial as it does not include all Lagrangian subspaces that intersect the Cauchy-data space $\cL_M$ transversally.
\end{remark}  

\begin{definition} 
	\label{laggrass2} 
	For $h \in H.$ Denote by
	\begin{equation} \label{eq:grassm2}
		\Gr_h(\cH') = \{P \in \cB(\cH')| P^2=P=P^*,\gamma P\gamma^*=\id-P\}
	\end{equation} the Lagrangian grassmannian of equivariant orthogonal projections on $\cH'.$ 
	Then, we define 
	\begin{align}
		\Gr^2_F(\cH'):=\Gr^2_{F,h}(\cH')&=\{(P,Q) \;|\; P,Q \in \Gr_h(\cH'), (P,Q) \text{ is a Fredholm pair} \} \label{eq:fredlagpair2} \\
		\Gr^2_*(\cH'):=\Gr^2_{*,h}(\cH')&=\{(P,Q) \;|\; P,Q \in \Gr_h(\cH'), (P,Q)  \text{ is an invertible pair} \label{eq:invlagpair2} \}.
	\end{align} 
	
	From \eqref{eq:laggrassisom1}, an induced associated map between $\Gr^2_F(\cH')$ and $\fL^2_F(\cH')$ is given by 
	\begin{equation} \label{eq:laggrassisom2}
		(P,Q) \mapsto (\ker P, \im Q),\quad\quad (\id-P_{\cL_1},P_{\cL_2}) \mapsfrom (\cL_1,\cL_2). 
	\end{equation}
\end{definition}

An immediate adaptation of Lemma~\ref{PQlemma1} to $\Gr_h(\cH')$ is as follows.

\begin{lemma}\label{PQlemma2}
	Let $h=\diag(a,WaW^*) \in H.$ Let $P,Q \in \Gr_h(\cH').$  Then,  $(P,Q)$ is an equivariant Fredholm (resp.  invertible)  pair if and only if $-1 \notin \spec_{ess}((a-\id)+aT^*S)$ (resp. $-1 \notin \spec((a-\id)+aT^*S)).$ There is an isomorphism  between the $h$-space
	$\ker(P) \cap \im(Q)$ and the space $\ker(a(\id + T^*S)).$  
\end{lemma}

\begin{definition} \label{equivmaslov}
	For $h \in H,$ let $\cL(t)=(\cL_1(t),\cL_2(t)): [0,1] \to \fL^2_F(\cH')$ be a continuous path of $h$-equivariant Fredholm Lagrangian pairs. Choose a good grid partition. Then, the $h$-equivariant Maslov index 
	$$
	\mas_h: \pi_1(\fL^2_F(\cH'))_H \to \C 
	$$
	of the path $\cL(t)$ is defined by
	\begin{equation}
		\mas_h(\cL_1(t),\cL_2(t)):= \sum^n_{j=1} \Big[ \tr \big( h|_{\cL_1(t_j) \cap \cL_2(t_j)}\big) 
		- \tr \big( h|_{\cL_1(t_{j-1}) \cap \cL_2(t_{j-1})}\big) \Big]. 
	\end{equation}
	
\end{definition}


\begin{lemma}
	For $h \in H,$ the equivariant Maslov index $\mas_h$ is well-defined.
\end{lemma}
\begin{proof}
	Suppose a good grid partition is chosen. Consider the grid over $[t_{j-1},t_j].$ Without loss of generality, we may assume that the paths $\cL_1(t)$ and $\cL_2(t)$ do not intersect transversally at the four vertices of the grid. Then, for any $t_d \in [t_{j-1},t_j],$ by a similar argument as in the proof of Proposition~\ref{spfwelld}, we have 
	\begin{align*}
		\mas_h(\cL_1(t),\cL_2(t))_{\substack{t_{j-1} \leq t \leq t_j}} 
		&= \Big(  \tr \big( h|_{\cL_1(t_j) \cap \cL_2(t_j)}\big) 
		- \tr \big( h|_{\cL_1(t_{d}) \cap \cL_2(t_{d})}\big) \Big) \\
		&+ \Big( \tr \big( h|_{\cL_1(t_d) \cap \cL_2(t_d)}\big) 
		- \tr \big( h|_{\cL_1(t_{j-1}) \cap \cL_2(t_{j-1})}\big)\Big) \\
		&= \tr \big( h|_{\cL_1(t_j) \cap \cL_2(t_j)}\big) 
		- \tr \big( h|_{\cL_1(t_{j-1}) \cap \cL_2(t_{j-1})}\big).
	\end{align*}
	Here, we include the possibility of paths whose ends do not meet the initial of the path in the next grid. In the case where a grid partition is fixed and the paths intersect transversally at some $t_d \in [t_{j-1},t_j],$ then we compute  $\mas_h$ at $t_d \pm\epsilon$ for some $\epsilon >0$ such that $\cL_1(t_d \pm\epsilon) \cap \cL_2(t_d \pm\epsilon) \neq 0.$ 
\end{proof}

We state the following properties of $\mas_h$ without proof. 
\begin{proposition}
	For $h \in H,$  $\mas_h$ satisfies the following properties.
	\begin{enumerate}
		\item (Reversal) Let $(\widetilde{\cL_1}(t),\widetilde{\cL_2}(t))$ be the Lagrangian paths $(\cL_1(t), \cL_2(t))$ traversing in the opposite direction. Then, 
		\[
		\mas_h(\widetilde{\cL_1}(t),\widetilde{\cL_2}(t))=-\mas_h(\cL_1(t),\cL_2(t)).
		\]
		\item (Homotopy invariance) Let $(\cL_1(t,s),\cL_2(t,s)): [0,1] \times [0,1] \to \fL^2_F(\cH')$ be a continuous path with $\cL_1(t,0)=\cL_1(t), \cL_1(t,1)=\cL'_1(t)$ and   $\cL_2(t,0)=\cL_2(t), \cL_2(t,1)=\cL'_2(t).$ Then,
		\[
		\mas_h(\cL_1(t),\cL_2(t))=\mas_h(\cL'_1(t),\cL'_2(t)).
		\]
		\item (Path additivity) Let $(\cL_1(t),\cL_2(t))$ and $(\cL'_1(t),\cL'_2(t))$ be two Lagrangian paths in $\pi_1(\fL^2_F(\cH'))_H$ such that $\cL_i(1)=\cL'_i(0)$ for $i=1,2.$ Then, $\mas_h$ is additive with respect to the concatenation of paths, i.e.
		\begin{align*}
			\mas_h((\cL_1(t)&*\cL'_1(t),\cL_2(t)*\cL'_2(t)))  \\ &= \mas_h((\cL_1(t),\cL_2(t))+ \mas_h(\cL'_1(t),\cL'_2(t))). \nonumber
		\end{align*}
		\item The map $\mas_h$ is natural under any $h$-equivariant isomorphism between $\cH'.$
		\item If $h$ acts trivially, then the $h$-equivariant Maslov index coincides with the classical Maslov index, i.e. $\mas_h=\mas.$ If, furthermore, $\cL_1(t)\cap \cL_2(t)$ is of constant dimension for all $t,$ then $\mas_h(\cL_1(t),\cL_2(t))=0.$
	\end{enumerate}
\end{proposition}


\section{Equivariant winding number} \label{sec5}


Let $\cH'$ be the $H$-module Hilbert space defined by \eqref{eq:Hprime}. Let $\cU(\cH')$ be the unitary group of $\cH'$. Let $\cU_F(\cH')$ (resp. $\cU_*(\cH')$) be the subset that consists of unitary operators $T \in \cU(\cH')$ satisfying $-1 \notin \spec_{ess}(T)$ (resp. $-1 \notin \spec(T)$).  They are not groups since for any unitary operators $T,S$ in $\cU_F(\cH')$ or in $\cU_*(\cH'),$ their product $TS$ does not necessarily lie in the set. On the other hand, let $\cU_\cK(\cH')$  (resp. $\cU_\tr(\cH')$) be the subgroup that consists of unitary operators $T \in \cU(\cH')$ such that  $T-\id \in \cK(\cH')$  (resp.  $T-\id$ is trace class).

\begin{lemma} 
	\label{laggrassisom3}
	For $h \in H.$ Let $\Gr^2_F(\cH')$ and $\Gr^2_*(\cH') $ be as in Definition \ref{laggrass2}. Then, there are diffeomorphisms
	\begin{equation} 
		\label{eq:laggrassisom3}
		\cU_F(\cE_i) \cong \Gr^2_F(\cH'), \quad \quad  \cU_*(\cE_i) \cong \Gr^2_*(\cH'). 
	\end{equation}
\end{lemma}
\begin{proof}
	Each Fredholm (resp. invertible) pair $(P,Q)$ is mapped to $T^*S,$ where $T=\Phi(P)$ and $S=\Phi(Q)$ in view of \eqref{eq:projform1}. In particular, $T^*S \in \cU_F(\cE_i)$ (resp. $T^*S \in \cU_*(\cE_i)$). Conversely, for any unitary operators $T : \cE_i \to \cE_{-i},$ we have $P=\frac{1}{2}\begin{psmallmatrix} \id & T^* \\ T & \id \end{psmallmatrix} \in \Gr_h(\cH').$ Similarly, for an unitary  $S:\cE_i \to \cE_{-1}$ we get $Q=\frac{1}{2}\begin{psmallmatrix} \id & S^* \\ S & \id\end{psmallmatrix} \in \Gr_h(\cH').$ One verifies that such a $(P,Q)$ is a Fredholm (resp. invertible) pair. 
\end{proof}

\begin{lemma}
	$\cU_\cK(\cH')$ is weakly $H$-homotopy equivalent to $\cU_F(\cH').$
\end{lemma}
\begin{proof}
	This is an equivariant adaptation of \cite[Lemma 6.1]{15}. Let $\pi: \cB(\cH') \to \cQ(\cH')$ be the quotient map as in Section 3. Then, $\pi : \cU_F(\cH') \to \cU_*\cQ(\cH')$ is an $H$-Hilbert bundle with fibres $\cU_\cK(\cH').$ The action of $H$ acts accordingly by the rule $(gh)(u)=u(g^{-1}h)$ for $g,h \in H$ and $u \in \cU_\cK,\cU_F$ and $\cU_*\cQ.$ By \cite[Theorem 3.5]{18}, $\cU_*\cQ(\cH') \simeq GL(\cQ(\cH'))$ is $H$-contractible.
\end{proof}

Since $\cU_\tr(\cH')$ is known to be homotopy equivalent to $\cU_\cK(\cH'),$ the following holds.
\begin{corollary}
	$\pi_1(\cU_\tr(\cH'))_H \cong \pi_1(\cU_\cK(\cH'))_H \cong \pi_1(\cU_F(\cH'))_H.$
\end{corollary}

Without loss of generality, consider an equivariant path of unitary operators $f:[0,1] \to \cU_F$ whose endpoints are in $\cU_*,$ i.e.  the path $f(t)$ is an element of  $\pi_1(\cU_F(\cH'),\cU_*(\cH'))_H,$ or equivalently, of $\pi_1(\cU_F(\cH'))_H$ by the $H$-contractibility of $\cU_*.$ If the endpoints of $f(t)$ are not in $\cU_*(\cH'),$ then we may perturb the endpoints $f(k)$ by $e^{-i\epsilon}$ for some $\epsilon >0$ such that $-1 \notin \spec(f(k)e^{-i\epsilon})$ for $k=0,1.$ This means that for $t \in [0,1],$ by a perturbation of the endpoints, a  general equivariant path $f(t)$ can be concatenated to another path $p(t) \in \cU_\ast(\cH')$ such that $p(0)=f(1)$ and $p(1)=f(0),$ forming a loop $f(t) \ast p(t).$ This concatenation is independent of the choice of $p(t).$ This suggests that the consideration of an equivariant loop can be relaxed a little to an equivariant path of unitary operators, with perturbed endpoints if necessary.

As inspired by  \cite[Lecture 10]{19}, we give the following definition.

\begin{definition}
	\label{equivFredDet}
	For $h \in H,$ define the equivariant Fredholm determinant by 
	\[
	\Det_h  : \pi_1(\cU_\tr(\cH'))_H \to \C
	\]
	\begin{equation}
		\label{eq:equivFredDet}
		\Det_h(f(t)):= \exp\Big(\int^1_0 \tr_h(f(t)^{-1} \frac{d}{dt}f(t)) dt\Big)
	\end{equation}
	where $\tr_h(f^{-1}df):= \tr(hf^{-1}df).$
\end{definition}

\begin{lemma} 
	\label{Fdetmulti}
	For $h \in H,$ the equivariant Fredholm determinant $\Det_h$ is multiplicative, i.e. $\Det_h(f_tg_t)=\Det_h(f_t)\Det_h(g_t).$
\end{lemma}

\begin{proof}
	Let $f_t=f(t),g_t=g(t) \in  \pi_1(\cU_\tr(\cH'))_H.$ By definition, we have
	$\tr_h((f_tg_t)^{-1} d(f_tg_t))$ $= \tr(h(f_tg_t)^{-1} d(f_tg_t)).$
	Then, by the conjugation invariance of trace, we have
	\begin{align*}
		\tr(h(f_tg_t)^{-1} d(f_tg_t))
		&= \tr(hg^{-1}_t f^{-1}_t [ (df_t)g_t+ f_t (dg_t)]) \\
		&= \tr(g^{-1}_t (hf^{-1}_t (df_t))g_t + hg^{-1}_t (dg_t)) \\
		&= \tr(hf^{-1}_t (df_t)) + \tr(hg^{-1}_t (dg_t)).
	\end{align*} By taking integration and exponential, the lemma follows.
\end{proof}

\begin{definition} 
	\label{equivwinding}
	For $h \in H,$ assume that a good grid partition is chosen. Define the equivariant winding number by
	\[
	w_h(f(t)): \pi_1(\cU_F(\cH'))_H \to \C
	\]
	\begin{equation}
		\label{eq:equivwind0}
		w_h(f(t))= \sum^n_{\substack{j=1 \\ 0 \leq \theta_{i} < \epsilon_{i} }} \left[ \tr \left( h|_{\ker\left(f(t_j)-e^{i(\pi+\theta_j)}\right)}\right) 
		- \tr \left( h|_{\ker\left(f(t_{j-1})-e^{i(\pi+\theta_{j-1})}\right)}\right) \right].
	\end{equation}
\end{definition}

In analogy to \cite[Proposition 5.7]{15}, we have an equivariant analog of some properties of winding numbers.
\begin{proposition}
	\label{windingprop0}
	For $h \in H,$  $w_h$ satisfies the following properties.
	\begin{enumerate}[(a)]
		\item (Reversal) Let $\widetilde{f}(t)$ be the path $f(t)$ traversing in the opposite direction. Then, 
		\[
		w_h(\widetilde{f}(t))=-w_h(f(t)).
		\]
		\item (Homotopy invariance) Let $f(t),f'(t) \in \pi_1(\cU_F(\cH'))_H$ be such that $f(0)=f'(0).$ Then, 
		\begin{equation}
			w_h(f(t))=w_h(f'(t)).
		\end{equation}
		\item (Path additivity) Let $f(t)=f_1(t) \ast f_2(t)$ be the concatenation of two paths. Then, 
		\begin{equation}
			w_h(f(t))=w_h(f_1(t))+w_h(f_2(t)).
		\end{equation}
		\item If $f\equiv f(t)$ for all $t,$ then, 
		\begin{equation}
			w_h(f)=0.
		\end{equation}	
		\item Let $f : [0,1] \to \cU_{\tr}(\cH')$ be a $C^1$-curve. Then,
		\begin{equation} 
			\label{eq:windingtracelog}
			w_h(f(t))= \frac{1}{2 \pi i} \Big( \int^1_0 \tr \log(hf(t)) dt - \tr \log(hf(1)) + \tr \log(hf(0))\Big).
		\end{equation}
	\end{enumerate}
\end{proposition}

\begin{theorem} 
	\label{maswind1}
	Let $h=\diag(a,WaW^*) \in H.$ Let $(\cL_1(t),\cL_2(t)) \in \pi_1(\fL^2_F(\cH'))_H.$ Let $T(t)$ and $S(t)$ be the associated unitary operators of $\cL_1(t)$ and $\cL_2(t)$ respectively. Then, 
	\begin{equation} 
		\label{eq:maswind1}
		\mas_h(\cL_1(t),\cL_2(t)) = w_h(T^*(t)S(t))=w(aT^*(t)S(t)) 
	\end{equation} where $w$ is the usual winding number.
\end{theorem}

\begin{proof}
	Let $(P(t),Q(t)) \in \pi_1(\Gr^2_F(\cH'))_H.$ By \eqref{eq:laggrassisom2}, we have the map $(\cL_1(t),\cL_2(t)) \mapsto (\id -P_{\cL_1(t)}, P_{\cL_2(t)}).$ Then, the claim follows from Definition \ref{equivmaslov}, Definition \ref{equivwinding} and Lemma \ref{PQlemma2} that there is an isomorphism between the $h$-space $\ker(P(t)) \cap \im(Q(t))$ and $\ker(a(\id+T^*(t)S(t)))$  for all $t.$
\end{proof}

Let $U \in \cU_F(\cH').$ Let $\cH_0:=\ker(U+\id)$ and $\cH_1:=\cH_0^{\perp}.$  Decompose $\cH'$ into the sum
\begin{equation}
	\label{eq:splitting0}
	\cH'=\cH_0\oplus \cH_1.
\end{equation}
Then, the unitary operator $U$ splits accordingly into $-\id_{\cH_0}\oplus \widetilde{U}$ where $-1\notin \text{spec}(\widetilde{U}).$ For $t\in [0,1],$ construct a contractible path by 
\[
f(t)=-\id_{\cH_0} \oplus \exp(t\log\widetilde{U}),
\] cf.\cite[Lemma 6.1(2)]{15}.  Let $h=\diag(a,WaW^*)$ where $a \in \cU(\cE_i).$ Consider an $h$-equivariant path $f_a(t)$ whose $h$-action is given by
\begin{equation}\label{eq:equivpath0}
	f_a(t)=a\odot f(t):=-a\:\id_{\cH_0}\oplus \exp(t\log (a\widetilde{U})),
\end{equation}
where $a$ is independent of $t\in [0,1].$ Similarly, by splitting $V \in \cU_\cK(\cH')$ into $-\id_{\cH_0}\oplus \widetilde{V}$ with respect to \eqref{eq:splitting0}, we consider
\begin{equation}
	\label{eq:equivpath1}
	g_a(t)=-a\:\id_{\cH_0}\oplus \exp(t\log (a\widetilde{V})).
\end{equation}
Then, by the ideal property of the group of compact operators $\cK=\cK(\cH'),$  the multiplication
\begin{equation}\label{eq:equivpath2}
	q_a(t):=a \odot (f(t) g(t)) =-a\id_{\cH_0} \oplus \exp(t \log (a \widetilde{U}\widetilde{V}))
\end{equation} is well-defined. This also works for $U \in \cU_\cK(\cH'), V \in \cU_F(\cH') $ or $U,V \in \cU_\cK(\cH').$ 

\vspace{0.5em}
The following is a slight generalisation of double index in \cite[Definition 6.2]{15}.

\begin{definition}
	Let $h=\diag(a,WaW^*) \in H.$ Let $U \in \cU_F(\cH')$ and $V\in \cU_\cK(\cH').$ Let $f(t) \in \pi_1(\cU_\cF(\cH'))_H$ and $g(t),q(t)  \in \pi_1(\cU_\cK(\cH'))_H.$ Define the $h$-equivariant double index of $U$ and $V$ by
	\begin{align}\label{eq:doubleind1}
		\tau_h &: \cU_\cF(\cH') \times \cU_\cK(\cH') \longrightarrow \C \nonumber \\
		\tau_h(U,V)&:= w_h(f(t)) + w_h(g(t))-w_h(q(t))
	\end{align} where $w_h(f(t)):=w(f_a(t))$ and the same for $g(t)$ and $q(t)$ respectively.
\end{definition}

\begin{proposition} 
	\label{doubleind2}
	\begin{enumerate}[(a)]
		\item If either one of the unitary operators is the identity $\id,$ then $$\tau_h(U,V)=0.$$
		\item If $V=U^{-1},$ then  
		\[
		\tau_h(U,U^{-1})= w_h(f(t))+w_h(f^{-1}(t)).
		\]
	\end{enumerate}	
\end{proposition}
\begin{proof}
	For (a), if both $U$ and $V$ are $\id,$ then $$\tau_h(\id,\id)=2w_h(\id)-w_h(\id)=w_h(\id)=w(a)=0$$ where the last equality follows from Proposition ~\ref{windingprop0}(d) as $a$ is independent of $t.$ 
	If $U=\id$ but $V \neq \id,$ then $q_a(t)=g_a(t)$ and  $\tau_h(\id,V)=w(g_a(t))-w(g_a(t))=0.$ The same applies when $V=\id$ but $U \neq \id.$
	For part (b), let $V=U^{-1}.$ Choose\footnote[2]{Once we fix both ends $g_a(0)=\id$ and $g_a(1)=U^{-1},$ this is possible to obtain by homotopy invariance.} $g_a(t)=f^{-1}_a(t)$ for all $t.$ Then, 
	\begin{align*}
		\tau_h(U,U^{-1}) &= w_h(f(t))+w_h(f^{-1}(t)) - w_h(f(t)f^{-1}(t)) \\
		&= w_h(f(t))+w_h(f^{-1}(t)) - w_h(\id) \\
		&= w_h(f(t))+w_h(f^{-1}(t)).  \qedhere
	\end{align*} 
\end{proof}

\begin{remark}
	\begin{enumerate}
		\item When the unitary operator $a$ is the identity operator, then the above definition and proposition reduce to the non-equivariant version of \cite{15}.
		\item All paths here are in the \textit{general position}. Hence, it is valid to consider the unitary operator $U \in \cU_F(\cH')$ as an endpoint of $f_a(t)$ due to the perturbation term $e^{ \pm i\theta}$ in \eqref{eq:equivwind0}. This would ensure that $f_a(t)\pm e^{i\theta}\cdot\id$ is compact.    
	\end{enumerate}
\end{remark}

\begin{corollary}
	Let $h=\diag(a,WaW^*) \in H.$ For $t \in [0,1],$ let $(\cL_0(t),\cL_1(t))$ be the pair of equivariant Lagrangian paths which correspond to the equivariant path of unitary operators $T^*(t)S(t).$ Let $U=T^*(1)S(1).$ Then, 
	\begin{equation}
		\mas_h(\cL_0(t),\cL_1(t))+\mas_h(\cL_1(t),\cL_0(t))= \tau_h(aU,aU^{-1}).
	\end{equation}
\end{corollary}
\begin{proof}
	By Theorem ~\ref{maswind1}, we have $\mas_h(\cL_1(t),\cL_0(t))=w_h(S^*(t)T(t)).$ Since both $T(t)$ and $S(t)$ are paths of unitary operators, we have $S^*(t)T(t)=(T^*(t)S(t))^{-1}.$ By Proposition \ref{doubleind2}(b), we obtain
	\begin{align*}
		\mas_h(\cL_0(t)&,\cL_1(t))+\mas_h(\cL_1(t),\cL_0(t)) \\
		&= w_h(T^*(t)S(t)) + w_h((T^*(t)S(t))^{-1}) 
		= \tau_h(U,U^{-1}). \qedhere
	\end{align*}  
\end{proof}
This tells us that, in general, the equivariant Maslov index of a Lagrangian pair is \textit{non-commutative}. For a non-trivial pair, this is commutative when the inverse path $g(t)$ is also the reversal of $f(t)=T^*(t)S(t),$ i.e. for all $t\in [0,1],$ 
\[
(T^*(t)S(t))^{-1}=T(1-t)S(1-t).
\] Explicitly, the reversal is given by 
\[
\overline{g}(t)
=-\id_{\cH_0}\oplus \exp((1-t) \log \widetilde{U}) 
=-\id_{\cH_0}\oplus \widetilde{U}\exp(-t \log \widetilde{U})
\]
and it is evidently to see that $q(t)$ reduces to the identity, which leads to the vanishing of $\tau_h(U,U^{-1})$ for $U=T^*(1)S(1).$ 
Next, since any two unitary operators can be connected to the identity operator by \eqref{eq:equivpath0}, we can therefore consider \textit{general} $h$-equivariant paths whose endpoints are not necessarily the identity operators. This allows us to define a double index \textit{relative} to the identity. 
\begin{definition}\label{doubleind1}
	Let $h\in H.$ For $t \in [0,1],$ let $f(t)$ and $g(t)$ be any $h$-equivariant continuous paths in $\cU_F(\cH')$ and $\cU_\cK(\cH')$ respectively. Define the \textit{relative $h$-double index} of $f(t)$ and $g(t)$ by 
	\begin{align} 
		\label{eq:doubleind0}
		\tau_h(f(t),g(t))&:=\tau_h(f(1),g(1))-\tau_h(f(0),g(0)) \\
		&=  w_h(f(t)) + w_h(g(t))-w_h(q(t)). \nonumber
	\end{align} 
\end{definition}

To see that this is well-defined, we consider the following skeleton path diagram.  

\begin{center}
	\includegraphics[width=.60\linewidth]{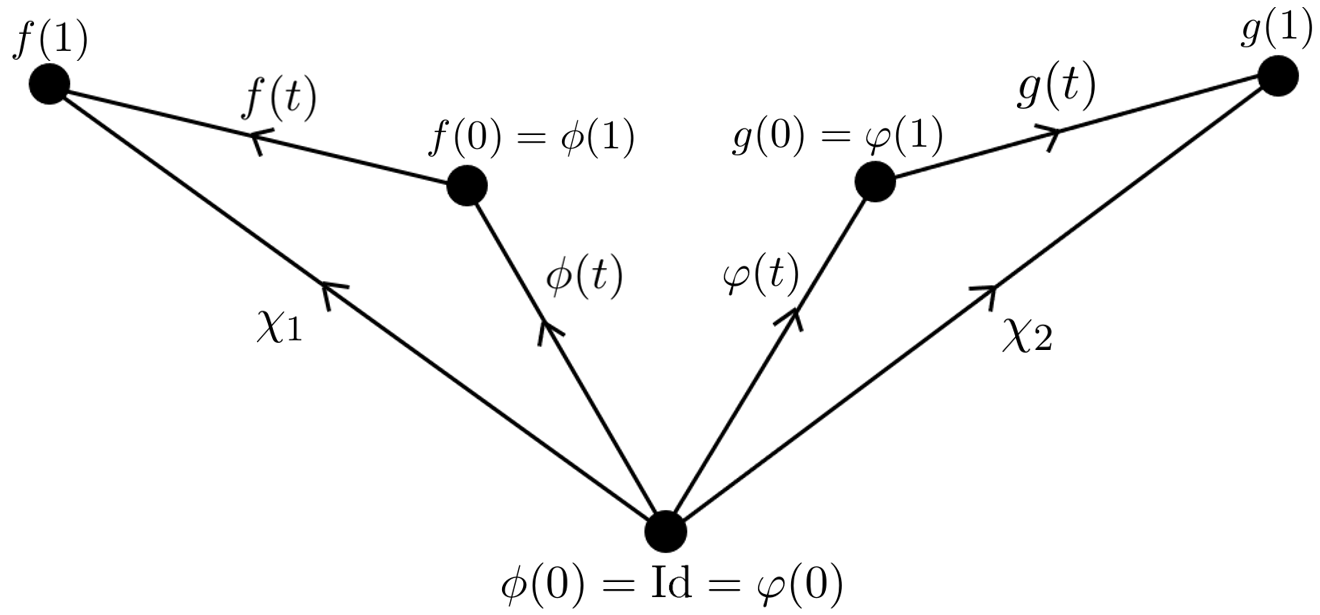}
\end{center}
Let $f$ and $g$ be \textit{any} $h$-equivariant paths as in Definition~\ref{doubleind1}. Let $\overline{\phi}$ and $\overline{\varphi}$ be the reversal of $\phi$ and $\varphi$ respectively. Then,  $f=\chi_1 \ast \overline{\phi}$ and $g=\chi_2 \ast \overline{\varphi}.$ By Proposition ~\ref{doubleind2}(c), we have 
\begin{align}
	&w_h(f)=w_h(\chi_1)+ w_h(\overline{\phi}),\quad w_h(g)=w_h(\chi_2)+ w_h(\overline{\varphi}), \label{eq:conpath1}\\
	&w_h(fg)=w_h((\chi_1 \ast \overline{\phi})(\chi_2 \ast \overline{\varphi}))= w_h(\chi_1\chi_2 )+ w_h(\overline{\phi}\overline{\varphi}). \label{eq:conpath2}
\end{align} The last equality can be seen as follows. Construct a path explicitly by
\begin{equation}\label{eq:concatenation1}
	(\chi_1 \ast \overline{\phi})(\chi_2 \ast \overline{\varphi})=
	\begin{cases}
		-\id \oplus \exp((1-2t) \log(\widetilde{\overline{\phi}}(0)\widetilde{\overline{\varphi}}(0))), \quad \quad 0 \leq t \leq \frac{1}{2}, \\
		-\id \oplus \exp((2t-1) \log(\widetilde{\chi}_1(1)\widetilde{\chi}_2(1))), \quad \frac{1}{2} \leq t \leq 1,
	\end{cases} 
\end{equation} where the four operators split into the form $r=-\id_{\cH_0} \oplus  \widetilde{r}$ such that $-1 \notin \text{spec}(\widetilde{r}).$ Then, we observe that the concatenation of the paths \eqref{eq:concatenation1} coincide with $(\chi_1\chi_2)\ast (\overline{\phi}\overline{\varphi}).$ By Proposition ~\ref{doubleind2}(c), we obtain \eqref{eq:conpath2}. 

Then, by \eqref{eq:conpath1},\eqref{eq:conpath2}, Proposition ~\ref{doubleind2}(a) and \eqref{eq:doubleind0}, we obtain
\begin{align*}
	\tau_h(f(1),g(1))
	&=w_h(f)+w_h(g)-w_h(fg) \\
	&= w_h(\chi_1)+ w_h(\overline{\phi})+ w_h(\chi_2)+ w_h(\overline{\varphi})+
	w_h(\chi_1\chi_2 )+ w_h(\overline{\phi}\overline{\varphi}) \\
	&= \tau_h(\chi_1(1),\chi_2(1)) - \tau_h(\phi(1),\varphi(1)) .
\end{align*} This justifies the term ``relative'' for the double index where $f$ and $g$ are general paths with possibly non-trivial endpoints. Now, we conclude some properties of the $h$-double index in the following, which are essentially an extension from that of the $h$-winding number $w_h$. They are also a slight generalisation of \cite[Proposition 6.3]{15}.

\begin{proposition}
	\label{doubleindproperties}
	For $h \in H,$ $\tau_h$ satisfies the following properties.
	\begin{enumerate}
		\item (Reversal) Let $\widetilde{f}(t)$ and $\widetilde{g}(t)$ be the respective paths $f(t)$ and $g(t)$ with both traversing in the opposite direction. Then, 
		\[
		\tau_h(\widetilde{f}(t),\widetilde{g}(t))=-\tau_h(f(t),g(t)).
		\]
		\item (Homotopy invariance)  For $i =0,1,$ let $f(t),f'(t)$ be two paths in $\cU_F(\cH')$ such that $f(i)=f'(i)$ and $g(t),g'(t)$ be two paths in $\cU_\cK(\cH')$ such that $g(i)=g'(i).$  Then, 
		\[
		\tau_h(f(t),g(t))=\tau_h(f'(t),g'(t)).
		\] 
		\item (Path additivity) Consider 
		\[
		f(t)= 
		\begin{cases}
			f_1(2t),\quad\quad \quad  0 \leq t \leq \frac{1}{2}, \\
			f_2(2t-1), \;\quad \frac{1}{2} \leq t \leq 1. 
		\end{cases}
		\]
		Let $f(t)=f_1(t) \ast f_2(t)$ and $g(t)=g_1(t) \ast g_2(t)$ be the concatenation of two paths in $\cU_F(\cH')$ and $\cU_\cK(\cH')$ respectively. Then, 
		\[
		\tau_h(f(t),g(t))= \tau_h(f_1(t),g_1(t)) + \tau_h(f_2(t),g_2(t)).
		\]
		\item Let $t \in [0,1].$ Let $f\equiv f(t)$ and $g\equiv g(t)$ be two constant paths in $\cU_F(\cH')$ and $\cU_\cK(\cH')$ respectively. Then, 
		\[
		\tau_h(f,g)=0.
		\]
	\end{enumerate}
\end{proposition}


\section{Equivariant Maslov triple index} \label{sec6}


Inspired by \cite{15}, we also consider an equivariant version of \textit{Maslov triple indices}. For simplicity, we write $P_t:=P(t), T_t:=T(t), \cL_i:=\cL_i(t),$ etc in the following.
\begin{definition} 
	\label{maslovtriple1}
	Let $h \in H.$ Let $P_t,Q_t,N_t$ be continuous paths in $\Gr_h(\cH')$ such that every pair $(P_t,Q_t),(Q_t,N_t)$ and $(P_t,N_t)$ are $h$-Fredholm and at least one of the differences $P_t-Q_t,Q_t-N_t$ and $P_t-N_t$ is compact. Let $T_t,S_t$ and $R_t$ be the associated unitary operators of $P_t,Q_t$ and $N_t$ respectively. Then, define the $h$-Maslov triple index $\tau^h_\mu$ of $(P_t,Q_t,N_t)$ by 
	\begin{equation} 
		\label{eq:maslovtriple2}
		\tau^h_\mu (P_t,Q_t,N_t) := \tau_h(T^*_tS_t,S^*_tR_t).
	\end{equation}
\end{definition}

This is well-defined relative to two endpoints as a consequence of the well-definedness of $\tau_h$ and thus inherits the properties in Proposition \ref{doubleindproperties}. Moreover, equivariant Maslov triple indices can be expressed as some mixed terms of equivariant Maslov double indices:
\begin{align} 
	\label{eq:maslovtriple3}
	\tau^h_\mu &(P_1,Q_1,N_1) - \tau^h_\mu (P_0,Q_0,N_0) \\
	&=  \tau_h(T^*_1S_1,S^*_1R_1) -\tau_h(T^*_0S_0,S^*_0R_0) \nonumber \\
	&= w_h(T^*_tS_t) + w_h(S^*_tR_t) - w_h(T^*_tR_t) \nonumber \\
	&= \mas_h(\cL_0,\cL_1) + \mas_h(\cL_1,\cL_2) - \mas_h(\cL_0,\cL_2).  \nonumber 
\end{align}

\begin{proposition} 
	\label{mastrisym}
	For $h \in H$, $\tau^h_\mu$ satisfies the following properties. 
	\begin{align}
		\tau^h_\mu(Q_t,P_t,N_t)
		&=-\tau^h_\mu(P_t,Q_t,N_t)+\tau_h(T^*_1S_1,(T^*_1S_1)^{-1})  \label{eq:triplemas1} \\
		\tau^h_\mu(P_t,N_t,Q_t)
		&=-\tau^h_\mu(P_t,Q_t,N_t)+\tau_h(S^*_1R_1,(S^*_1R_1)^{-1})  
		\label{eq:triplemas2} \\
		\tau^h_\mu(N_t,Q_t,P_t)
		&=-\tau^h_\mu(P_t,Q_t,N_t)+ \tau_h(T^*_1S_1,(T^*_1S_1)^{-1})  \label{eq:triplemas3} \\ 
		&\quad\qquad +\tau_h(S^*_1R_1,(S^*_1R_1)^{-1}) -\tau_h(T^*_1R_1,(T^*_1R_1)^{-1}).
		\nonumber
	\end{align}
\end{proposition}

\begin{proof}
	For \eqref{eq:triplemas1}, we obtain by a direct computation
	\begin{align*}
		\tau^h_\mu (P_t,Q_t,N_t) + \tau^h_\mu (Q_t,P_t,N_t) 
		&=  \tau_h(T^*_tS_t,S^*_tR_t) +\tau_h(S^*_tT_t,T^*_tR_t)  \\
		&= w_h(T^*_tS_t) + w_h(S^*_tR_t) - w_h(T^*_tR_t)  \\
		&\quad\quad + w_h(S^*_tT_t) + w_h(T^*_tR_t) - w_h(S^*_tR_t) \\
		&= \tau_h(T^*_1S_1,(T^*_1S_1)^{-1}). 
	\end{align*} 
	Equalities \eqref{eq:triplemas2} and \eqref{eq:triplemas3} can be obtained similarly, so we leave them to readers.
\end{proof}

\begin{corollary}
	Let $h \in H.$
	\begin{enumerate}[(a)]
		\item  If $P_t=Q_t$ for all $t,$ then \[\tau^h_\mu(P_t,P_t,N_t)=0.\]
		\item  If $Q_t=N_t$ for all $t,$ then \[\tau^h_\mu(P_t,Q_t,Q_t)=0.\]
		\item  If $P_t=N_t$ for all $t,$ then 
		\[\tau^h_\mu(P_t,Q_t,P_t)= \tau_h(T^*_1S_1,(T^*_1S_1)^{-1}).\]
		\item If $P_t=Q_t=N_t$ for all $t,$ then \[\tau^h_\mu(P_t,P_t,P_t)= 0.\]
	\end{enumerate}
\end{corollary}


\section{Equivariant $\zeta$-determinants} \label{sec7}


\begin{definition} \label{zeta1}
	Let $h \in H.$ Define the $h$-equivariant $\zeta$-function of a self-adjoint elliptic equivariant differential operator $D=D_{P}$ of degree $d>0$ by
	\begin{equation} \label{eq:zeta1}
		\zeta_h(D,s) := \tr(hD^{-s}) = \sum_{\lambda \in \spec(D)\backslash \{0\}} \tr(h^*_\lambda) \lambda^{-s}
	\end{equation} where $h^*_\lambda$ is the induced linear map on the $\lambda$-eigenspace.
\end{definition}

\begin{definition}
	Let $\{\lambda, \phi_\lambda\}$ be a spectral decomposition of $D=D_P.$ Let \[
	K(t,x,y)= \sum_\lambda e^{-\lambda t} \phi_\lambda(x) \otimes \bar{\phi}_\lambda(y)
	\] 
	be the smooth kernels of the heat operator $e^{-tD}.$ Let $h \in H$ for which the induced map $h^*$ acts with respect to the second component $(x,hy),$ then
	\begin{equation}
		\tr(h^*K(t,x,hy))= \sum_\lambda e^{-\lambda t} \phi_{\lambda}(x) \otimes \bar{\phi}_\lambda (hy).
	\end{equation} 
	Define the corresponding $h$-kernel by
	\begin{equation}
		K(h,t,x,y):=\int_M \tr(h^*K(t,x,hy)) d\text{vol}(y) 
	\end{equation} 
	and the $h$-trace by
	\begin{equation}
		\tr(he^{-tD}):=\int_M \tr(h^*K(t,x,hx))d\text{vol}(x) =\sum_k \tr(h^*_{\lambda_k}) e^{-\lambda_k t}.
	\end{equation}
\end{definition}

\begin{remark}
	It was shown in \cite{14} and \cite[\S 6.3]{6} that the heat kernel $K(h,t)$ has an asymptotic expansion as $t \to 0.$ For $x$ not in the fixed point set of $h,$ the heat kernel $K(h,t,x,x)$ decays rapidly as $t \to 0.$ By an argument of partition of unity, the trace $\tr(he^{-tD})$ has an asymptotic expansion as $t \to 0$ of the form 
	\begin{equation} \label{eq:asympexp1}
		\tr(he^{-tD}) \sim  t^{-n/d} ( a_0(D)+ a_1(D)t^{1/d}+\cdots + a_N(D) t^{N/d} + O( t^{(N+1)/d})) 
	\end{equation} where $d$ is the order of the $h$-equivariant elliptic differential operator $D$ and each $a_i(D)$ is the integral of some function on an $n$-dimensional manifold $M^n.$
\end{remark}

\begin{lemma}
	For $h \in H,$ the $h$-equivariant $\zeta$-function of $D$ is the Mellin transform of $\tr(he^{-tD}).$
\end{lemma}

\begin{proof}
	For $\lambda_{p_i} \leq \lambda_{p_{i+1}} \leq \cdots < 0 < \lambda_1 \leq \lambda_2 \leq \cdots,$ the trace $\tr(he^{-tD})$ is given by the sum $\sum^\infty_{k=p_i} \tr(h^*_{\lambda_k}) e^{-\lambda_k t}.$ Set $\tau = \lambda t.$  Then, we compute
	\begin{align*}
		\int^\infty_{0} t^{s-1} \tr(h^*_{\lambda}) e^{-\lambda t} dt
		&= \int^\infty_{0} \frac{\tau^{s-1}}{\lambda^{s-1}} \tr(h^*_{\lambda}) e^{-\tau} \lambda^{-1} d\tau 
		= \frac{\tr(h^*_{\lambda})}{\lambda^{s}} \int^\infty_{0} \tau^{s-1} e^{-\tau}  d\tau  \\
		&= \tr(h^*_{\lambda})\lambda^{-s} \Gamma(s).
	\end{align*} Thus, 
	\begin{align*}
		\zeta_h(D,s) &= \sum_{k} \tr(h^*_{\lambda_k}) \lambda^{-s}_k 
		= \frac{1}{\Gamma(s)} \sum_{k} \int^\infty_{0} t^{s-1} \tr(h^*_{\lambda_k}) e^{-\lambda_k t} dt \\
		&= \frac{1}{\Gamma(s)} \int^\infty_{0} t^{s-1} \sum_{k} \Big(\tr(h^*_{\lambda_k}) e^{-\lambda_k t} \Big) dt  
		=  \frac{1}{\Gamma(s)} \int^\infty_{0} t^{s-1} \tr(he^{-tD}) dt.
	\end{align*} 
\end{proof}

In contrast to $\eta$-functions, the $\zeta$-function is often defined for unbounded operators such as the square of Dirac operators whose discrete spectrum is positive. In this case, the appropriate heat trace is the \textit{renormalised heat trace} 
$$\tr(he^{-tD|_{P^+}})= \sum^\infty_{k=1} \tr(h^*_{\lambda_k}) e^{-\lambda_k t} = \sum^\infty_{k=p_i} \tr(h^*_{\lambda_k}) e^{-\lambda_k t} -\sum^{0}_{k=p_i} \tr(h^*_{\lambda_k}) e^{-\lambda_k t}.$$ 

\begin{corollary}
	For $h \in H,$ the $h$-equivariant $\zeta$-function of the operator $D|_{P^+}$ is given by 
	\begin{equation} \label{eq:zeta2}
		\zeta_h(D,s) =\sum^\infty_{k=1} \tr(h^*_{\lambda_k}) \lambda_k^{-s} = \frac{1}{\Gamma(s)} \int^\infty_{0} t^{s-1} \tr(he^{-tD|_{P^+}}) dt.
	\end{equation}
\end{corollary}

\begin{remark} 
	In this case, the renormalised heat trace also has an asymptotic expansion as $t\to 0$ which is in general different than that of \eqref{eq:asympexp1}. In particular, the corresponding coefficient $\tilde{a}_i(D^2)$
	takes the form
	$$\tilde{a}_i(D) = a_i(D) - \frac{1}{r!}\sum^0_{k=n_p}(-\tr(h^*_{\lambda_k}) \lambda_k)^r$$ when $r= (i-n)/d \in \{0,1,2,\ldots\}.$ The existence and an explicit description of the expansion were shown by Donnelly \cite[Theorem 4.1]{12}. Moreover, by a slight generalisation of Gilkey's result \cite{14}, the residue of $\Gamma(s)\zeta_h(D,s)$ at $s=(n-i)/d$ is
	\begin{equation} 
		\text{Res}_{s=(n-i)/d} (\Gamma(s)\zeta_h(D,s)) =\tilde{a}_i(D).
	\end{equation} 
	At $s=0,$ the gamma function is known to have a Laurent expansion 
	\begin{equation} \label{eq:gammafn1}
		\Gamma(s)=\frac{1}{s}+ \gamma + p(s)
	\end{equation}
	where $\gamma$ is the Euler constant and $p(s)$ is the part of the expansion with positive powers of $s.$ As it has no pole at $s=0,$ the function $\zeta_h(D,s)$ is regular at $s=0.$ The constant term is
	\begin{equation}
		\zeta_h(D):= \zeta_h(D,0)=\text{Res}_{s=0} (\Gamma(s)\zeta_h(D,s))=\tilde{a}_n(D)
	\end{equation} for $n=\dim M.$ We shall call $\zeta_h(D)$ the $h$-equivariant $\zeta$-invariant of $D.$
\end{remark}

\begin{proposition} 
\label{estimate1}
Let $P \in \Gr_h^\infty(A).$ Let $P^\partial \in \Gr_h^\infty(A)$ be the perturbed APS projection as given by \eqref{eq:perturbedAPS}. Then, there exists constants $c_1,c_2>0$ such that 
	\begin{equation}
		|\tr(he^{-tD^2_P})-\tr(he^{-tD^2_{P^{\partial}}})| \leq c_1\sqrt{t}\tr|\cK|e^{c_2t||\cK||}.
	\end{equation}
\end{proposition}
\begin{proof}
	This is a slight generalisation of \cite[Corollary 2.2, Theorem 3.2]{27}. Assume that a fixed $h \in H$ commutes with $\cK_2=g^{-1}[A,g]$ where $g$ is a unitary operator on $L^2(M,S)$ acting as the identity on the interior $M\backslash ([0,1]\times \partial M)$ and as a path of certain unitaries in $\Gr_\infty(A)$ on the cylinder $[0,1]\times \partial M,$  c.f. \cite[Lemma 1.2,(1.10)]{27}. Note that $h$ is independent of $t \in [0,1]$ and the commutation does not affect each factor of the local expression of heat kernels, we observe that the method of Duhamel's principle still works, in which as a consequence of \cite[Proposition 1.1,2.1,3.2]{27} we have an upper estimate
	\begin{equation}
		\label{eq:estimate2}
		|\tr(he^{-tD^2_P})-\tr(he^{-tD^2_{P^{\partial}}})| \leq c\cdot \tr|h^*|\sqrt{t}\tr|\cK|e^{c_0t|h|||\cK||} 
	\end{equation} where $\tr(h^*)=\sum_k \tr(h^*_{\lambda_k})$ is the trace of the induced linear map.
\end{proof}

It follows that we have an equivariant adaptation of the classically known fact regarding the independence of the choice of boundary projections for $\zeta$-invariants of $D^2,$ see \cite[Proposition 0.5]{27}. 

\begin{lemma}
	\label{zetaindependence}
	For $h \in H,$	the value of the $h$-equivariant $\zeta$-function at $s=0$ is constant on $\Gr_h^\infty(A),$ i.e. for any $P,Q \in \Gr_h^\infty(A),$ we have
	\begin{equation}
		\zeta_h(D^2_{P})=\zeta_h(D^2_{Q}).
	\end{equation}
\end{lemma}
\begin{proof} 
	For $P,Q,P^\partial \in \Gr_h^\infty(A),$ 
	\begin{align*}
		\zeta_h(D^2_{P}) - \zeta_h(D^2_{P^\partial}) 
		&= \lim_{s \to 0} \frac{1}{\Gamma(s)} \int^\infty_0 t^{s-1}\tr (he^{-tD^2_P}- he^{-tD^2_{P^\partial}})dt \\
		&=\lim_{s \to 0} s  \int^\infty_0 t^{s-1}\tr (he^{-tD^2_P}- he^{-tD^2_{P^\partial}})dt \quad\quad \text{from } \eqref{eq:gammafn1}
	\end{align*} 
	By the estimate \eqref{eq:estimate2}, we have 
	\begin{align*}
		|\zeta_h(D^2_{P}) - \zeta_h(D^2_{P^\partial})  |	
		&\leq \lim_{s \to 0} s  \int^\infty_0 t^{s-1} |\tr (he^{-tD^2_P}- he^{-tD^2_{P^\partial}})|dt \\
		&\leq c \cdot \lim_{s \to 0} s  \int^\infty_0 t^{s-\frac{1}{2}} dt =0
	\end{align*}
	where the second inequality follows from Proposition ~\ref{estimate1}. Lastly, the claim follows  $|\zeta_h(D^2_{P}) - \zeta_h(D^2_{Q})|\leq |\zeta_h(D^2_{P}) - \zeta_h(D^2_{P^\partial}) | + |\zeta_h(D^2_{P^\partial})- \zeta_h(D^2_{Q})| \leq 0.$	
\end{proof}

\begin{proposition}
	\label{zetadet1}
	Let $h\in H.$ Let $D^2$ be the square of a self-adjoint $h$-equivariant Dirac operator $D.$ Then, 
	\begin{equation}
		\dot{\zeta}_h(D^2) =  \int^\infty_0 \frac{1}{t}\tr(he^{-tD^2}) dt -\frac{\tilde{a}_n(D^2)}{s}\Big|_{s=0} - \tilde{a}_n(D^2)\cdot \gamma
	\end{equation} 
\end{proposition}

\begin{proof}
	We adopt an approach of Singer in \cite{25}. Let $p(s)= \int^\infty_0 t^{s-1}\tr(he^{-tD^2}) dt = \frac{\tilde{a}_n(D^2)}{s} + b + f(s)$ where $b$ is some constant and $f(s)$ consists of all the terms with positive powers of $s.$  Then, \eqref{eq:zeta2} can be written as 
	$$\zeta_h(D^2,s)=\frac{p(s)}{\frac{1}{s}+ \gamma + p(s)}= \frac{sp(s)}{1+ s\gamma + p_1(s)}$$ where we use the expansion \eqref{eq:gammafn1}.  Its derivative at $s=0$ can be computed as
	\begin{align*}
		\dot{\zeta}_h(D^2) =\dot{\zeta}_h(D^2,0) 
		&= \frac{d}{ds}\Big|_{s=0} \zeta_h(D^2,s) = \frac{d}{ds}\Big|_{s=0} (sp(s))(1+ s\gamma + p_1(s))^{-1} \\
		&= \frac{s\dot{p}(s)+p(s)}{1+ s\gamma + p_1(s)} \Big|_{s=0} - \frac{(sp(s))(s\gamma+s\dot{p}_1(s))}{(1+ s\gamma + p_1(s))^2}\Big|_{s=0} \\
		&= \int^\infty_0 \frac{1}{t}\tr(he^{-tD^2}) dt -\frac{\tilde{a}_n(D^2)}{s}\Big|_{s=0} - \tilde{a}_n(D^2)\cdot \gamma
	\end{align*}
	where $p(0)=\int^\infty_0 \frac{1}{t}\tr(he^{-tD^2}) dt$   and $\zeta_h(D^2)=\tilde{a}_n(D^2).$
\end{proof}

\begin{remark}\label{zetadet2}
	The proof of Proposition~\ref{zetadet1} extends without change to the operator $D|_{P^+}$ which is exactly the derivative of $\zeta_h(D|_{P^+})$ of \eqref{eq:zeta2} with respect to the second equality.
	In this case, another description of $\dot{\zeta}_h(D|_{P^+}) $ can be obtained from the first equality of \eqref{eq:zeta2}. More precisely, 
	\begin{align} 
		\dot{\zeta}_h(D|_{P^+})
		&= \frac{d}{ds}\Big|_{s=0} \zeta_h(D|_{P^+},s) 
		= \sum^\infty_{k=1} \tr(h^*_{\lambda_k})\frac{d}{ds}\Big|_{s=0} (\lambda_k)^{-s} \\
		&= - \sum^\infty_{k=1} \tr(h^*_{\lambda_k}) \log(\lambda_k) 
		= - \sum^\infty_{k=1} \log \Big(\lambda_k^{\tr(h^*_{\lambda_k})} \Big)  \nonumber \\
		&= - \log \Bigg( \prod^\infty_{k=1} \lambda^{\tr(h^*_{\lambda_k})}_k  \Bigg) \nonumber.
	\end{align}
\end{remark}

\begin{definition} \label{zetadet3}
	Let $h \in H.$ Define the $h$-equivariant \textit{regularised} $\zeta$-determinant of $D=D_P$ by 
	\begin{equation}
		\Det^h_\zeta(D|_{P^+})= e^{-\dot{\zeta}^{h}_{D}(0)}.
	\end{equation}
\end{definition}

In view of Remark~\ref{zetadet2} and Definition~\ref{zetadet3}, it was observed by Singer \cite{25} that it is possible to extend to the case of self-adjoint elliptic operators whose symbols are not necessarily positive definite. In particular, such an operator can have positive $\{\lambda_k\}$ or negative $\{-\mu_k\}$ eigenvalues. Then, the corresponding $\zeta^h$-function can be written as 

\begin{align} \label{eq:zeta3}
	\zeta^h_D(s)
	= \sum_k &\tr(h^*_{\lambda_k}) \lambda^{-s}_k + \sum_k (-1)^{-s} \tr(h^*_{\mu_k})\mu^{-s}_k   \\
	=\sum_k &\Big(\frac{\tr(h^*_{\lambda_k})\lambda^{-s}_k -\tr(h^*_{\mu_k}) \mu^{-s}_k}{2} + \frac{\tr(h^*_{\lambda_k}) \lambda^{-s}_k + \tr(h^*_{\mu_k}) \mu^{-s}_k}{2}\Big)  \nonumber \\
	&+ (-1)^{-s} \sum_k  
	\Big( \frac{\tr(h^*_{\lambda_k})\lambda^{-s}_k + \tr(h^*_{\mu_k})\mu^{-s}_k}{2} - \frac{\tr(h^*_{\lambda_k})\lambda^{-s}_k -\tr(h^*_{\mu_k}) \mu^{-s}_k}{2}\Big). \nonumber
\end{align}
Let $\eta^h_D(s)= \sum_k \tr(h^*_{\lambda_k}) \lambda^{-s}_k -\tr(h^*_{\mu_k}) \mu^{-s}_k$ be the equivariant $\eta$-function of $D.$ (Compare with Definition~\ref{equiveta1}). Let 
\begin{align*}
	\zeta^h_{D^2}(s/2) 
	&= \sum_k \tr(h^*_{\lambda_k})  (\lambda^2)^{-\frac{s}{2}}_k + \tr(h^*_{\mu_k}) (\mu^2)^{-\frac{s}{2}}_k \\
	&= \sum_k \tr(h^*_{\lambda_k})  |\lambda_k|^{-s} + \tr(h^*_{\mu_k}) |\mu_k|^{-s} \\
	&= \sum_k \tr(h^*_{\lambda_k})  \lambda^{-s}_k + \tr(h^*_{\mu_k}) \mu^{-s}_k
\end{align*}
be the equivariant $\zeta$-function of $D^2$ of the variable $s/2.$ (Compare with Definition~\ref{zeta1}). Then, \eqref{eq:zeta3} can be rewritten as 
\begin{equation}\label{eq:zeta4}
	\zeta^h_D(s)= \frac{\eta^h_{D}(s)+\zeta^h_{D^2}(s/2)}{2} + (-1)^{-s} \Bigg( \frac{\zeta^h_{D^2}(s/2) -\eta^h_{D}(s) }{2} \Bigg).
\end{equation}

Choose $(-1)^{-s}=e^{-i\pi s}.$ Then, the derivative $\dot{\zeta}^h_D(0)$ of \eqref{eq:zeta4} is 
\begin{align*}
	\dot{\zeta}^h_D(0) 
	&= \frac{\dot{\eta}^h_D(0)}{2} + \frac{\dot{\zeta}^h_{D^2}(0)}{4} -\frac{i\pi(-1)^{-s}}{2}\big(\zeta^h_{D^2}(s/2) -\eta^h_{D}(s)\big)|_{s=0}  \\
	&\quad + (-1)^{-s}\Big(\frac{\dot{\zeta}^h_{D^2}(0)}{4} - \frac{\dot{\eta}^h_D(0)}{2}\Big)|_{s=0} \\
	&= \frac{\dot{\zeta}^h_{D^2}(0)}{2} - \frac{i \pi}{2}(\zeta^h_{D^2}(0)- \eta^h_{D}(0)).
\end{align*}

\begin{corollary} 
	\label{corzetadet1}
	Let $h \in H.$ The $h$-equivariant $\zeta$-determinant of $D$ is given by
	\begin{equation}\label{eq:{corzetadet1}}
		\Det^h_\zeta(D):= \exp \Big(\frac{i \pi}{2}(\zeta_h(D^2)- \eta_h(D))-\frac{1}{2}\dot{\zeta}_h(D^2)\Big).
	\end{equation}
\end{corollary}


\section{Equivariant Scott-Wojciechowski Theorem} \label{sec8}


Next, we conjecture an equivariant adaptation of the Scott-Wojciechowski Theorem which, roughly speaking, allows us to express the ratio of the $h$-equivariant $\zeta$-determinants of the Dirac operators (associated with two boundary projecitons) as the Fredholm determinant of the respective unitaries associated to the projections.  The original version was first introduced by Scott and Wojciechowski \cite{24}. Using this, Kirk and Lesch proved a remarkable relation between the difference of the $\eta$-invariants of Dirac operators (associated with two different boundary projections) and the Fredholm determinant of the corresponding unitary operators, see \cite[\S 4]{15}.  

To begin with, we propose a suitable modification to the settings in \cite{24}. The main reason for such a modification is due to the convention that we have used in our definitions above such as the equivariant winding number $w_h$ and the equivariant double index $\tau_h$. For instance, in our case the unitary operator under the $h$ action is taken to be  $aT^*S$ acting from $\cE_i$ to $\cE_i;$ whilst in the case of \cite{24} the unitary operator is acting from $\cE_{-i}$ to $\cE_{-i}.$ If this was the case, then it would imply the unitary operators are acted on by $WaW^*,$ which creates unwanted complication. In the following, we \textit{do not intent} to prove the equivariant version of the Scott-Wojciechowski Theorem in details, but only to lay out some suggestions of our modification to their settings. Therefore, we will leave out most of the original details out for the interested readers.

Let $\cH'=L^2(H,\cH)$ be our model Hilbert space as defined by \eqref{eq:Hprime}. Let $\cF(\cH')$ be the space of equivariant Fredholm operators on $\cH'.$ Let $\cF_0(\cH')$ be the connected component of index zero. Fix $B \in \cF_0(\cH').$ Set 
$$\cF_B=\cF_B(\cH')=\{S \in \cF(\cH') \;|\; S-B \text{ is trace-class}\}.$$ Fix a trace-class operator $\mathscr{B}$ such that $S=B+\mathscr{B}$ is invertible. Let $h \in H,$ define the determinant line bundle of $B$ by 
$$\text{Det}(B)=\cF_B \times_{\sim_h} \C$$ where the equivalence relation $\sim_h$ incorporated with the $h$-action is given by 
\[
(R,z)_h:=h \cdot (R,z)=(S(S^{-1}R),z)_h=(aS(S^{-1}R),z)_{1_H} \sim (S, z \cdot  \Det_F(aS^{-1}R))_{1_H}
\] for $h \in H, R,S \in \cF_B$ and $z \in \C.$ Here, $\cdot$ denotes the $h$-action; $1_H$ denotes the identity element in $H$; the third equality is the action by $h=\diag(a,WaW^*)$ on $S(S^{-1}R)$ and the last equivalence follows from the multiplicative property of Fredholm determinant $\Det_F.$ Let $[R,z]_h$ be the equivalence class of $(R,z)_h.$ Define the determinant element by 
\[
\Det(B):=[B,1]_{1_H},
\]
and observe that $\Det(B) \neq 0$ if and only if $B$ is invertible. Here, the last entry denotes the identity element $1=e^{2\pi i} \in \C.$ 

Recall from \eqref{eq:grinfty1} that $\Gr_h^\infty(A)$ is the space of all equivariant boundary projections which differ from the APS projection $P^+$ by a smoothing operator.  We identify the Grassmannian  $\Gr_h^\infty(A)$ with the group $\cU_\infty(\cE_i)$ of unitary operators that differ from $\id_{\cE_i}$ by a smoothing operator as follows. Let $\cP_M \in \Gr_h^\infty(A)$ be the canonical equivariant Calder\'on projection on $M$. Let $K:C^{\infty}(\partial X, E_i) \to C^{\infty}(\partial X, E_{-i})$  be the associated unitary operator such that the Cauchy-data space $\cL_M$ (cf. \eqref{eq:CDsp}) coincides with $\text{graph}(K).$ Let $P \in \Gr_h^\infty(A)$ be any projection. Let $T: \cE_{i} \to \cE_{-i}$ be the associated unitary operator such that 
\[
\text{Ran}(P)=\text{graph}(T^{-1}).
\] 
In particular, the map $P \mapsto T^{-1}K$ defines an isomorphism $\Gr_h^\infty(A) \cong \cU_{\infty}(\cE_i)$ where $\cU_{\infty}(\cE_i)=\{V \in \cU_F(\cE_i)\;|\; V-\id \text{ is smoothing}\}.$ Note that the commutativity of $h$ and $P \in \Gr_h^\infty(A)$ imply $a \in  \cU_\infty (\cE_i)$ and $WaW^* \in \cU_\infty (\cE_{-i}).$ 

Let $$V(P)=\begin{pmatrix} K^{-1}T & \textbf{0} \\ \textbf{0} & \id_{\cE_{-i}} \end{pmatrix}, \quad \text{ and so }
V^{-1}(P)=\begin{pmatrix} T^{-1}K & \textbf{0} \\ \textbf{0} & \id_{\cE_{-i}} \end{pmatrix}.$$ Then, one can check that in terms of the homogeneous structure of the Grassmannian
$$
P=V^{-1}(P)\cP_m V(P)
$$ where $P$ and $\cP_M$ are expressed in terms of \eqref{eq:projform1} with the unitaries $T$ and $K$ respectively. Consider the composition  $\cS(P)=P\cP_M : L_M \to \text{Ran}(P).$ Then, \cite[Lemma 1.1]{24} reads as follows.

\begin{lemma}
	$V(P)$ differs from $\cS(P)$ by a smoothing operator, i.e. $V(P)-\cS(P) \in \Psi^{-\infty}.$ 
\end{lemma}
\begin{proof}
	The proof here is modified from that of \cite[Lemma 1.1]{24} by using our new setting. The operator $V(P)$ acts by its inverse from $\cL_M=\text{graph}(K)$ to $\text{Ran}(P)=\text{graph}(T^{-1})$ 
	$$
	\begin{pmatrix}
		x \\ Kx
	\end{pmatrix} \longrightarrow 
	\begin{pmatrix}
		T^{-1}K & \textbf{0} \\ \textbf{0} & \id_{\cE_{-i}}
	\end{pmatrix}
	\begin{pmatrix}
		x \\ Kx
	\end{pmatrix}
	=\begin{pmatrix}
		T^{-1}(Kx) \\ Kx
	\end{pmatrix}
	=\begin{pmatrix}
		T^{-1}y \\ y
	\end{pmatrix}.
	$$ By a direct computation, the composition $\cS(P)$ can be written as 
	$$\cS(P)=P\cP_M=
	\frac{1}{4}\begin{pmatrix}
		\id+T^{-1}K & T^{-1}+ K^{-1} \\ T+K & \id+ TK^{-1}
	\end{pmatrix}.
	$$ Then, 
	\begin{align*}
		\cS(P)\begin{pmatrix} x \\ Kx \end{pmatrix}
		&=\frac{1}{2} \begin{pmatrix} (\id+T^{-1}K)x \\ (\id + TK^{-1})Kx \end{pmatrix} \\
		&= \frac{1}{2} \begin{pmatrix} (\id+K^{-1}T)(T^{-1}y) \\ (\id + TK^{-1})y \end{pmatrix} \\
		&= \begin{pmatrix} \frac{\id+K^{-1}T}{2} & \textbf{0} \\ \textbf{0} & \frac{\id + TK^{-1}}{2} \end{pmatrix}
		\begin{pmatrix} T^{-1}y \\ y \end{pmatrix}.
	\end{align*} This shows that with respect to $\text{Ran}(P)$ the composition $\cS(P)$ is given by  $\diag(\frac{\id+K^{-1}T}{2},\frac{\id + TK^{-1}}{2}).$ The rest of the claim follows from the fact that $K^{-1}T$ (resp. $TK^{-1}$) differs from $\id_{\cE_{i}}$ (resp. $\id_{\cE_{-i}}$) by a smoothing operator.
\end{proof}

As a consequence, $\det(V(P))=[V(P),1]_{1_H}$ defines an element of the determinant line bundle $\text{Det}_h(\cS(P)).$ For $h \in H,$ let $B : \cL_M \to \text{Ran}(P)$ be an invertible $h$-equivariant Fredholm operator such that $B-\cS(P)$ is a trace class operator. Then, 
\begin{align*}
	\Det(B)
	&=[B,1]_{1_H} 
	=[V(P)(V^{-1}(P)B),1]_{1_H} 
	=[V(P),1\cdot \Det_F(V^{-1}(P)B)]_{1_H}  \\
	&=\Det_F(V^{-1}(P)B) \cdot[V(P),1]_{1_H} 
	=\Det_F(V^{-1}(P)B) \cdot \Det(V(P)).
\end{align*}

\begin{definition}
	For $h \in H$, define the $h$-canonical determinant of $D_P$ by 
	\begin{equation} \label{eq:equivcandet}
		\Det^h_\cC (D_P):=\Det^h_\cC \cS(P) := \Det_F(hV^{-1}(P)\cS(P)).
	\end{equation}
\end{definition}
It is straightforward to verify that  
\[
hV^{-1}(P)\cS(P)= \diag \Big( \frac{a(\id + T^{-1}K)}{2}, \frac{WaW^*(\id + TK^{-1})}{2} \Big).
\]
The determinant on the right side of \eqref{eq:equivcandet} is the Fredholm determinant of the diagonal matrix acting on the graph of $K.$ Hence, 
\[
\Det^h_\cC (D_P) = \Det_F \Big( \frac{a(\id + T^{-1}K)}{2}\Big).
\]
(Compare with \cite[Lemma 1.3]{24}). To see the well-definedness on the right side, we rewrite 
\begin{align*}
	\frac{a(\id + T^{-1}K)}{2} 
	&= \frac{\id + aT^{-1}K - (\id -a)}{2} \\
	&= \id + \frac{ [aT^{-1}K - (\id -a)] - \id}{2}.
\end{align*}
Since $\id-a$ is smoothing for $a \in \cU_{\infty}(\cE_{i})$, so is $[aT^{-1}K - (\id -a)] - \id.$ Hence, $[aT^{-1}K - (\id -a)] - \id$ is of trace class and  $\frac{1}{2} a(\id + T^{-1}K)$ is of determinant class. This argument also works if we replace $K$ by any other projection $S \in \Gr^{\infty}_h(A).$

An equivariant version of the Scott-Wojciechowski Theorem reads as follows.

\begin{conjecture}\label{equivsctwoj}
	For $h\in H,$ the following equality holds over $\Gr_h^\infty(A)$
	\begin{equation}
		\Det^h_\zeta (D_P) = \Det^h_\zeta (D_{\cP_M}) \cdot \Det_F \Big( \frac{a(\id + T^{-1}K)}{2}\Big).
	\end{equation}
\end{conjecture}

We state this as a conjecture instead of a theorem. The basis of believing that the conjecture is true are (i) the equivariant modification above is compatible with that of Scott and Wojciechowski, so the proof is only expected to subject to minor changes; (ii) the formula of  equivariant $\zeta$-determinants given in Corollary ~\ref{corzetadet1} turns out to have the same form as the one they are working with.

By using our settings above, we give a subcase of an equivariant version of \cite[Theorem 4.2]{15} as follows. 

\begin{theorem}
Let $h \in H.$ Let $P,Q \in \Gr_h^\infty(A)$ and $T=T(P),S=S(Q)$ be the respective associated unitary operators such that $(P,\cP_M)$ and $(\cP_M, Q)$ are invertible pairs. Then,
	\begin{equation}\label{eq:fredddetexp}
		\exp(2\pi i (\tilde{\eta}_h(D_P)- \tilde{\eta}_h(D_Q)))=\Det_F(aT^*S),
	\end{equation}
 where $\tilde{\eta}_h(D_P)$ and $\tilde{\eta}_h(D_Q)$ are the reduced equivariant $\eta$-invariants of $D_P$ and $D_Q$ respectively (see $\S 9$).
\end{theorem}
\begin{proof}
We adopt an elegant method introduced by Kirk and Lesch \cite[Theorem 4.2]{15}. First, 
consider the $h$-equivariant invertible pair $(\cP_M,Q),$ for which the corresponding boundary problems $D_{\cP_M}$ and $D_Q$ are both invertible. 
Let $K=K(\cP_M)$ be the unitary operator associated to $\cP_M$. By Lemma~\ref{zetaindependence}, we know that the equivariant $\zeta$-invariant is independent of the projections, i.e.  $\zeta_h(D^2_{\cP_M})=\zeta_h(D^2_Q).$ By Conjecture ~\ref{equivsctwoj} and Corollary ~\ref{corzetadet1}, we have 
	\begin{align*}
		\Det_F \Big( \frac{a(\id + K^{*}S)}{2}\Big)
		&= \frac{\Det^h_\zeta (D_Q)}{\Det^h_\zeta (D_{\cP_M})} 
		= \frac{\exp \Big(\frac{i \pi}{2}(\zeta_h(D^2_{Q})- \eta_h(D_{Q}))-\frac{1}{2}\dot{\zeta}_h(D^2_{Q})\Big)}{\exp \Big(\frac{i \pi}{2}(\zeta_h(D^2_{\cP_M})- \eta_h(D_{\cP_M}))-\frac{1}{2}\dot{\zeta}_h(D^2_{\cP_M})\Big)} \\
		&= \exp \Big(\frac{i \pi}{2}(\eta_h(D_{\cP_M})- \eta_h(D_{Q}))+\frac{1}{2}(\dot{\zeta}_h(D^2_{\cP_M})-\dot{\zeta}_h(D^2_{Q}))\Big).
	\end{align*}
	By taking the modulus on both sides, we see that $$\Big|\Det_F \Big( \frac{a(\id + K^{*}S)}{2}\Big)\Big|=\exp \Big(\frac{1}{2}(\dot{\zeta}_h(D^2_{\cP_M})-\dot{\zeta}_h(D^2_{Q}))\Big)$$ and thus   
	\begin{equation}\label{eq:Freddivide1}
		\Det_F \Big( \frac{a(\id + K^{*}S)}{2}\Big)\Big/\Big|\Det_F \Big( \frac{a(\id + K^{*}S)}{2}\Big)\Big|
		=\exp \Big(\frac{i \pi}{2}(\eta_h(D_{\cP_M})- \eta_h(D_{Q})\Big).
	\end{equation}
	On the other hand, by Lemma~\ref{Fdetmulti}, we have 
	\[
	\Det_F\Big(a\Big(\frac{\id+K^*S}{2} \Big)\Big) = \Det_F(a) \Det_F\Big(\frac{\id+K^*S}{2} \Big).
	\]
	The right side is well-defined because $a=\id +(a-\id)$ for $a \in \cU_\infty(\cH'),$ so $a-\id$ is smoothing and is of trace class. Hence, $a$ is of determinant class. Similarly, 
	\[
	\frac{\id + K^*S}{2} = \id + \frac{K^*S-\id}{2}.
	\] Since $K^*S-\id$ is smoothing and is of trace class, $\frac{1}{2}(K^*S-\id)$ is of determinant class. 
	Then, 
	\begin{align*}
		\Det_F\Big(a\Big(\frac{\id+K^*S}{2} \Big)\Big) &= \Det_F(a) \Det_F(K^*S)\Det_F\Big(\frac{\id+S^*K}{2} \Big) \\
		&= \Det_F(aK^*S) \Det_F\Big(\frac{\id+S^*K}{2} \Big).
	\end{align*}
	Note that $\frac{1+S^*K}{2}$ is in general not unitary and hence its Fredholm determinant is a complex number whose modulus is a non-zero real number. By taking modulo, we obtain
	\begin{equation}\label{eq:Freddivide2}
		\Det_F\Big(a\Big(\frac{\id+K^*S}{2} \Big)\Big)\Big/ \Big|\Det_F\Big(a\Big(\frac{\id+K^*S}{2} \Big)\Big) \Big| = \Det_F(aK^*S).
	\end{equation}
	From \eqref{eq:Freddivide1} and \eqref{eq:Freddivide2}, we have
	\[
	\exp \Big(\frac{i \pi}{2}(\eta_h(D_{\cP_M})- \eta_h(D_{Q})\Big)=\Det_F(aK^*S).
	\] 
    Then,  
	\begin{align*}
		\exp \Big(\frac{i \pi}{2}(\eta_h(D_{P})- \eta_h(D_{Q})\Big) 
		&=  \exp \Big(\frac{i \pi}{2}(\eta_h(D_{P})- \eta_h(D_{\cP_M})\Big)  \exp \Big(\frac{i \pi}{2}(\eta_h(D_{\cP_M})- \eta_h(D_{Q})\Big) \\ 
		&= \Det_F(aT^*K)\Det_F(aK^*S) \\
		&= \Det_F(aT^*S)
	\end{align*} where the third line follows from the multiplicative property of equivariant Fredholm determinants. The invertibility of $D_P$ and $D_Q$ imply that there are no contributions by the $h$-traces and thus proving \eqref{eq:fredddetexp}.
\end{proof}


\section{Equivariant $\eta$-invariants} \label{sec9}


In this section, we discuss an equivariant version of the APS $\eta$-invariants and its relation with other equivariant invariants. The notion of equivariant $\eta$-invariants was classically well-studied, for instance, see \cite[pp 892]{12b} and \cite[\S 3]{12c}.

\begin{definition} 
\label{equiveta1}
(\cite{12b})
For $h \in H,$ the $h$-equivariant $\eta$-function of a self-adjoint elliptic equivariant differential operator $D$ is defined by
	\begin{equation}
		\eta_h(D,s) := \sum_{\lambda \neq 0, \lambda \in \spec(D)\backslash \{0\}} \tr(h^*_{\lambda}) \frac{\text{sgn}(\lambda)}{|\lambda|^s} 
	\end{equation} where $h^*$ is the induced linear map on the $\lambda$-eigenspace. The equivariant $\eta$-invariant $\eta_h(D)$ is defined to be $\eta_h(D,0).$
\end{definition}	

\begin{lemma}
	For $h \in H,$  $\eta_h(D,s)$ is the Mellin transform of $\tr(hD e^{-tD^2}).$
\end{lemma}

\begin{proof}
	Recall that the $\Gamma$-function with the variable $\frac{1}{2}(s+1)$ is given by 
	\begin{equation}
		\Gamma \Big(\frac{s+1}{2} \Big) = \int^\infty_{0} t^{\frac{s-1}{2}} e^{-t} dt.
	\end{equation} 
	Set $\tau = \lambda^2 t.$ Then, 
	\begin{align*}
		\int^\infty_0 t^{\frac{s-1}{2}} \tr(h^*_{\lambda}) \lambda e^{-\lambda^2 t} dt 
		&= \int^\infty_0 \Big( \frac{\tau}{\lambda^2}\Big)^{\frac{s-1}{2}}  \tr(h^*_{\lambda}) \frac{\lambda}{\lambda^2} e^{-\tau} d\tau   
		= \lambda^{-1}  \int^\infty_0 \frac{|\lambda|}{|\lambda|^s} \tr(h^*_{\lambda}) \tau^{\frac{s-1}{2}} e^{-\tau} d\tau  \\
		&= \Big(\int^\infty_0 \frac{|\lambda|}{|\lambda|^s} \tr(h^*_{\lambda}) \tau^{\frac{s-1}{2}} e^{-\tau} d\tau \Big) \frac{|\lambda|}{\lambda} |\lambda|^{-s}  \\
		&= \Gamma \Big(\frac{s+1}{2} \Big) \tr(h^*_{\lambda}) \text{sgn}(\lambda) |\lambda|^{-s}.
	\end{align*} 
	
	Thus, 
	\begin{align*}
		\eta_h(D,s)
		&= \frac{1}{\Gamma \big(\frac{s+1}{2} \big)} \sum_k \int^\infty_0 t^{\frac{s-1}{2}}  \tr(h^*_{\lambda_k}) \lambda e^{-\lambda^2_k t} dt  \\
		&= \frac{1}{\Gamma \big(\frac{s+1}{2} \big)} \int^\infty_0 t^{\frac{s-1}{2}}  \sum_k \Big( \tr(h^*_{\lambda_k}) \lambda e^{-\lambda^2_k t} \Big) dt \\
		&=   \frac{1}{\Gamma \big(\frac{s+1}{2} \big)} \int^\infty_0 t^{\frac{s-1}{2}}  \tr(hD e^{-tD^2})dt.
	\end{align*}
\end{proof}

\begin{definition}
	Let $h \in H.$ Let $\epsilon >0.$ Define the truncated $\eta$-function of $D$ by
	\begin{equation}\label{eq:truncatedeta}
		\eta_{h, \geq \epsilon}(D,s)= \frac{1}{\Gamma \big(\frac{s+1}{2} \big)} \int^\infty_\epsilon t^{\frac{s-1}{2}}  \tr(hD e^{-tD^2})dt.
	\end{equation}
\end{definition}

\begin{definition}
\label{reducedeta}
	Let $h \in H.$ Define the \textit{reduced equivariant $\eta$-invariant} of $D$ by
	\begin{equation}
		\tilde{\eta}_{h}(D)= \frac{\eta_{h}(D)+ \tr(h|_{\ker D})}{2}\quad \in \C.
	\end{equation}
\end{definition}

We now give an equivariant version of Getzler's spectral flow formula in terms of reduced $\eta$-invariants. See also \cite[Lemma 3.3]{15}.

\begin{theorem} 
	\label{equivGetSpecFlow}
	Let $h \in H.$ For $t \in [0,1],$ let $P_t$ be a smooth path in $\Gr(A_t)$ and let $D_P(t):=D_{P(t)}(t)$ be a smooth path of self-adjoint equivariant Dirac operators. Then,
	\begin{equation}
		\spf_h\{D_P(t)\}_{t \in [0,1]}  = \tilde{\eta}_h(D_P(1)) - \tilde{\eta}_h(D_P(0)) -\frac{1}{2} \int\frac{d}{dt} \eta_h(D_P(t)).
	\end{equation}
 \vspace{-2em}
\end{theorem}

\begin{proof}
	Note that our definition of equivariant spectral flow are based on the approach of Phillips \cite{22}, which is another approach to that of  \cite{3}. This coincides with the context that Getzler is working on, cf. \cite[Definition 2.3]{13}. Hence, the equivariant notion in Definition~\ref{equivspecflowdef} extends to Getzler's setting. In the following we only highlight the relevant changes to Getzler's setting without reciting all details. For $h \in H,$ consider the associated equivariant one-form 
	\[
	\alpha^h_\epsilon(X) = \sqrt{\frac{\epsilon}{\pi}} \tr(hXe^{-\epsilon D^2})
	\]
	for some  self-adjoint bounded linear operator $X$ on $\cH'=L^2(H,\cH).$ 
	Then, $\alpha^h_\epsilon$ has locally bounded first derivatives and is closed ($d\alpha^h_\epsilon=0)$. Moreover, $d\eta_{h,\epsilon}(D_P)$ can be derived to be $2\alpha^h_\epsilon$ on $\Phi_0=\{D\:|\: \dim \ker(D)=0\}.$  Then, the integral $\int d\eta_{h,\epsilon}$ is invariant under isotopies of the smooth paths $D(t).$ In particular, there is an isotopy from $D_P(t)$ to a path transversal to $\Phi_1=\{D\:|\: \dim \ker(D)=1\},$ so we may assume that $D_P(t)$ is transversal to $\Phi_1.$ Let $\{t_i\}^k_{i=1}$ be the ordered sequence of points in the open interval $(0,1)$ for which $D_P(t)$ intersects $\Phi_1.$ By \cite[Proposition 2.5]{13}, $(D_P(t))^*\eta_{h,\epsilon}$ is a smooth function in $t \in [0,1]$ except at the points $t_i.$  With the truncated $\eta-$function 
	\[
	\eta_{h,\epsilon}(\lambda)=\frac{1}{\sqrt{\pi}} \int^\infty_\epsilon t^{-\frac{1}{2}} \tr(h \lambda e^{t \lambda^2} ) dt 
	\] that changes signs across the $x-$axis, we have 
	\begin{align} \label{eq:getzlerspec1}
		\spf_h\{D_p(t)\}_{t\in[0,1]} 
		&= \frac{1}{2} \sum^k_{i=1} \Big( \lim_{t \to t^+_i} \eta_{h,\epsilon}(D_P(t)) - \lim_{t \to t^-_i} \eta_{h,\epsilon}(D_P(t)) \Big) \\
		&= \frac{1}{2}\Big( \eta_{h,\epsilon}(D_P(1)) - \eta_{h,\epsilon}(D_P(0)) -  \int_{D_P(t)} d\eta_{h,\epsilon}(D_P(t)) \Big)  \nonumber \\
		&= \frac{1}{2}\Big( \eta_{h,\epsilon}(D_P(1)) - \eta_{h,\epsilon}(D_P(0)) -  \int^1_0 \frac{d}{dt}\eta_{h,\epsilon}(D_P(t)) dt \Big)   \nonumber
	\end{align} 
	where the second line follows from the Fundamental Theorem of Calculus. From Definition~\ref{equiveta1},  for some $\epsilon>0,$ we can rewrite the equivariant $\eta$-function $\eta_h(D_P(t),s)$  as
	\begin{align}
		\eta_h(D_P(t))
		&= \sum_{\substack{\lambda \in \spec(D_p(t))\\ 0<|\lambda|<\epsilon}} \tr(h^*_\lambda) \frac{\text{sgn}(\lambda)}{|\lambda|^s} 
		+ \sum_{\substack{\lambda \in \spec(D_p(t))\\ \epsilon \leq |\lambda|}} \tr(h^*_\lambda) \frac{\text{sgn}(\lambda)}{|\lambda|^s} \\
		&= \eta_{h,<\epsilon}(D_P(t),s) + \eta_{h,\geq\epsilon}(D_P(t),s) \nonumber
	\end{align}
	for $\Re(s)$ sufficiently large. Here, $\eta_{h,\geq\epsilon}$ is the truncated $\eta$-invariant from \eqref{eq:truncatedeta}. Note that over $(-\epsilon,\epsilon)$ the $\eta$-invariant takes the form  
	\begin{equation}\label{eq:truncatedeta2}
		\eta_{h,<\epsilon}(D_P(t))=\sum_{\substack{\lambda \in \spec(D_p(t))\\ 0<|\lambda|<\epsilon}} \tr(h^*_\lambda) \text{sgn}(\lambda) = \sum_{\substack{\lambda \in \spec(D_p(t))\\ -\epsilon < t < \epsilon}} \pm \tr(h|_{\ker(P^+(t))})
	\end{equation} 
	where $P^+(t)$ denotes the positive spectral projection over $(-\epsilon,\epsilon),$  see \eqref{eq:hspec1} in Definition~\ref{equivspecflowdef}.
 
\noindent Then, by the additivity of spectral flow, we obtain
	\begin{equation}
		\label{eq:truncatedeta3}
		\eta_{h,<\epsilon}(D_P(1))-\eta_{h,<\epsilon}(D_P(0))= \tr(h|_{\ker(D_P(0))}) - \tr(h|_{\ker(D_P(1))}). 
	\end{equation} 
	From the first two terms on the right side of \eqref{eq:getzlerspec1}, \eqref{eq:truncatedeta2}, and \eqref{eq:truncatedeta3}, we deduce that
	\begin{align}
		\label{eq:truncatedeta4}
		\eta_{h,\geq\epsilon}(D_P(1))-\eta_{h,\geq\epsilon}(D_P(0))
		&= \big[ \eta_{h}(D_P(1))-\eta_{h}(D_P(0))\big] -  \big[\eta_{h,<\epsilon}(D_P(1))-\eta_{h,<\epsilon}(D_P(0))\big] \\
		&= \big[ \eta_{h}(D_P(1))+\tr(h|_{\ker(D_P(1))})\big] -\big[\eta_{h}(D_P(0)) +  \tr(h|_{\ker(D_P(0))}) \big] \nonumber.
	\end{align} 
	Together with the fact that $\eta_{h,>\epsilon}(D_P(t))$ is smooth and  $\eta_{h,>\epsilon}(D_P(t))\equiv \eta_h(D_P(t)) \mod R(H),$ we obtain from \eqref{eq:getzlerspec1} and \eqref{eq:truncatedeta4} that
	\begin{align*}
		\spf_h\{D_p(t)\}_{t\in[0,1]} 
		&= \frac{1}{2}\Big( \eta_{h,\epsilon}(D_P(1)) - \eta_{h,\epsilon}(D_P(0)) -  \int d\eta_{h,\epsilon}(D_P(t)) \Big)  \\
		&= \frac{1}{2}\big[ \eta_{h}(D_P(1))  + \tr(h|_{\ker(D_P(1))})\big] - \frac{1}{2}\big[\eta_{h}(D_P(0)) +  \big[\tr(h|_{\ker(D_P(0))}) \big] \\ &\quad\quad -  \frac{1}{2}\int d\eta_{h}(D_P(t))  \\
		&= \tilde{\eta}_h(D_P(1)) - \tilde{\eta}_h(D_P(0)) -  \frac{1}{2}\int d\eta_{h}(D_P(t)) .
	\end{align*} 
	This completes the proof.
\end{proof}

The next lemma is, strictly speaking, a special case of \cite[Theorem 4.4]{15} in which the unitary operator associated to a boundary projection takes an explicit form that uses error functions. Let 
\[
\varphi(x)=\frac{2}{\sqrt{\pi}} \int^x_0 e^{-t^2} dt
\]
be the error function of $x.$ Note that $\varphi \in C^1(\R)$ is a special case of normalizing functions satisfying 
\[
\varphi^{-1}(0)=\{0\}, \quad \lim_{x \to -\infty} \varphi(x) =-1, \quad  \lim_{x \to \infty} \varphi(x)=1,
\] 
\[ 
\varphi' \in C_0(\R) \text{ with } \varphi'(0)>0 \text{ and } \varphi'\geq 0.
\] See \cite{26}.  
For any $P \in \Gr_h(A),$ consider the associated unitary operator $T=\exp(i\pi \varphi(D_P)).$ From \eqref{eq:projform1}, there is an involution $\sigma =2P -\id$ for any given $P,$ where $T$ sits at the bottom left block and $T^*$ sits at the upper right block of $\sigma,$  satisfying the anti-commutativity property $\sigma\gamma=-\gamma \sigma.$

\begin{lemma} \label{etadiffunitary}
	Let $h \in H.$ For $t \in [0,1],$ let $D_P(t)$ be a smooth path of equivariant Dirac operators with invertible endpoints $D_P(0)$ and $D_P(1).$ 
	Then, 
	\begin{equation}
		\tilde{\eta}_h(D_{P}(1)) -\tilde{\eta}_h(D_{P}(0)) 
		=\frac{1}{2\pi i} \tr_h \log (T^*K) 
	\end{equation} 
	where $T=\exp(i\pi \varphi(D_P(1)))$ and  $K=\exp(i\pi \varphi(D_P(0)))$ for $\varphi(D_P(i))$ denote the error function evaluated at $D_P(i)$ with $i=0,1.$
\end{lemma}

\begin{proof} 
	Define $V_t=-\exp(i \pi \varphi(D_P(t))).$ Then, for $t \in [0,1],$
	\[
	V_t^{-1}=-\exp(-i \pi \varphi(D_P(t))),\quad\frac{dV_t}{dt}  
	=-i \pi \exp(i \pi \varphi(D_P(t))) \big(\frac{d}{dt}D_P(t)\big) \varphi'(D_P(t)).
	\] Hence, 
	\begin{equation}
		\label{eq:windingeta1}
		\tr_h \left(V_t^{-1}\frac{dV_t}{dt} \right) 
		= \frac{2\pi i}{\sqrt{\pi}} \tr \left(h\dot{D}_P(t) e^{-D^2_P(t)} \right).
	\end{equation} 
	On the other hand, from the proof of Theorem~\ref{equivGetSpecFlow}, which follows from \cite{13}, that with $\epsilon=1$ we obtain 
	\begin{equation}
		\label{eq:vareta1}
		\frac{d}{dt} \eta_h(D_p(t))=\frac{2}{\sqrt{\pi}} \tr(h\dot{D}_P(t)e^{-D^2_P(t)}).
	\end{equation} 
	Hence, from \eqref{eq:windingeta1} and \eqref{eq:vareta1}, we have 
	\begin{equation}
		\label{eq:vareta2}
		\frac{1}{2}\frac{d}{dt} \eta_h(D_p(t))
		= \frac{1}{2\pi i} \tr_h \left(V_t^{-1}\frac{dV_t}{dt} \right).
	\end{equation} 
	Consider the unitary path 
	$U_t = \exp(i \pi [t\varphi(D_P(0)) + (1-t)D_P(1)]).$  
	By a direct computation, we obtain
	\begin{equation}
		\label{eq:vareta3}
		\tr_h \left(U_t^{-1} \frac{dU_t}{dt}\right) 
		=\tr_h \log \exp(i \pi(\varphi(D_P(0))-\varphi(D_P(1))) )
	\end{equation} 
	where the right side is independent of $t.$ Let $\tilde{D}_P(t):= tD_P(0) + (1-t)D_P(1)$ be the linear path of equivariant Dirac paths associated to $U_t.$ One can check that $\tilde{D}^2_P(t)$ equals $D^2_P(1)$ when $t=0$ and $D^2_P(0)$ when $t=1.$ By applying \eqref{eq:vareta2} to $U_t$ and from \eqref{eq:vareta3}, we have 
	\begin{align*}
		\label{eq:vareta4}
		\frac{1}{2}[\eta_h(D_P(1))-\eta_h(D_P(0))]
		&= \frac{1}{2}\int^1_0 d\eta_h(\tilde{D}_P(t)) 
		= \frac{1}{2\pi i}\int^1_0  \tr_h \left(U_t^{-1} \frac{dU_t}{dt}\right)dt \\
		&= \frac{1}{2\pi i}\int^1_0 \tr_h \log \exp(i\pi(\varphi(D_P(0))  -\varphi(D_P(1))))dt \\ 
		&= \frac{1}{2\pi i}\tr_h \log (T^*K). 
	\end{align*}  
	By the invertibility assumption of $D_P(i)$ for $i=0,1,$ there is no contribution from the $h$-trace on the trivial kernels. Hence, we obtain 
	\begin{align*}
		\tilde{\eta}_h(1)-\tilde{\eta}_h(0)
		=\frac{1}{2}[\eta_h(&D_P(1)) + \tr(h|_{\ker D_P(1)})] \\
		&-\frac{1}{2}[\eta_h(D_P(0)) + \tr(h|_{\ker D_P(0)})] 
		=\frac{1}{2\pi i}\tr_h \log (T^*K).
	\end{align*}
\end{proof}

\begin{theorem} 
	\label{specmas}
	Let $h \in H.$ For $t \in [0,1],$ let $D_P(t)$ be a smooth path of equivariant Dirac operators with endpoints $D_P(0)=D_P$ and $D_P(1)=D_Q.$ Let $(\cL(t),\cL_M(t))$ be a smooth path in $\fL_F(\cH')$ whose endpoints $\cL(0)$ (resp. $\cL(1)$) correspond to the orthogonal projections $P$ (resp. $Q$). Then,
	\begin{equation}
		\spf_h\{D_P(t)\}= \mas_h(\cL(t),\cL_M(t)).
	\end{equation}
\end{theorem}

\begin{proof}
	First, let $P(t) \in \Gr_h^\infty(A(t)).$ Since the path $D_{\cP_M}(t):=D_{\cP_M(t)}(t)$ is invertible for all $t,$ its equivariant spectral flow is $\spf_h\{D_{\cP_M}(t)\}_{t \in [0,1]}=0.$ Let 
	\[
	T(t)=\exp(i\pi \varphi(D_P(t))), \quad K(t)=\exp(i\pi \varphi(D_{\cP_M}(t)))
	\]
	where $\varphi$ denotes the error function. By Theorem~\ref{equivGetSpecFlow}, Lemma~\ref{etadiffunitary},  \eqref{eq:windingtracelog}, and Theorem~\ref{maswind1}, we have 
	\begin{align*}
		\spf_h\{D_{P}(t)\}
		&= \spf_h\{D_{P}(t)\} - \spf_h\{D_{\cP_M}(t)\} \\
		&= \left[\tilde{\eta}_h(D_P(1)) - \tilde{\eta}_h(D_{\cP_M}(1)) \right] - \left[\tilde{\eta}_h(D_P(0)) - \tilde{\eta}_h(D_{\cP_M}(0)) \right] \\
		&\quad\quad - \frac{1}{2}\int^1_0 \frac{d}{dt}\left( \tilde{\eta}_h(D_P(t)) -\tilde{\eta}_h(D_{\cP_M}(t))  \right) dt \\
		&= \left[\frac{1}{2\pi i}\tr_h \log (T^*(1)K(1)) \right]- \left[\frac{1}{2\pi i}\tr_h \log (T^*(0)K(0))\right] \\
		& \quad\quad -\frac{1}{2 \pi i} \int^1_0 \frac{d}{dt}\left( \tr_h \log (T^*(t)K(t)) \right) dt \\
		&= w_h(T^*(t)K(t))               
		= \mas_h(\cL(t),\cL_M(t)).
	\end{align*} 
	Once this is established, the generalization to arbitrary $P(t) \in \Gr_h(A(t))$ can be done by using a similar trick (with appropriate modification) in the proof of \cite[Theorem 7.5]{15}. 
\end{proof}

Let $M^\pm$ be the compact spin manifold with boundary obtained by cutting $M$ into two disjoint parts along the boundary $Y=\partial M.$ Assume an identification between the boundary of the split manifolds $M^\pm$ and their collar neighborhoods $\partial M^+ \times (-\epsilon,0]$ and $ \partial M^- \times [0,\epsilon)$ respectively. Near the collar neighborhood, the Dirac operator takes the product form 
\[
D=\begin{pmatrix}
	\gamma & 0 \\ 0 & -\gamma
\end{pmatrix}
\left( 
\frac{d}{dx} +
\begin{pmatrix}
	A & 0 \\ 0 & -A
\end{pmatrix} 
\right).
\]
Since $h$ commutes with $\gamma$ and $A,$ it commutes with $D$ in this splitting. Hence, it extends to act on the components of $L^2(E|_{\partial M^{\pm}})=L^2(E|_{Y}) \oplus L^2(E|_{Y})$ respectively. The following is regarding the splitting of equivariant $\eta$-invariants with respect to $M^{\pm}.$ The proof can be obtained almost verbatim by the arguments of \cite{15}, which require a special angle-valued homotopy $P(\theta,P)$ of projections. Its spectral flow is not computable in general, but it leads to a complete splitting formula of $\eta$-invariants. Kirk and Lesch also established the equality $\spf\{D^{M^{\pm}}_{P(\theta,P)}\}= \spf\{D^{M^{\pm}}_{P_t \oplus (\id-P_t)}\}$ with respect to the splitting. By adapting the equivariance assumption by elements of $H$ above to their method, we obtain the following. 

\begin{theorem}
	\label{etasplitting1}
	Let $h \in H.$ Let $M=M^+ \cup_{Y} M^-.$ Let  $P_t \in \Gr_h(A_t)$ be an equivariant path connecting $\cP_{M^+}$ at $t=0$ and $P$ at $t=1.$ Similarly, let $\id -P_t$ be another path connecting $\cP_{M^-}$ and $\id-P.$  Let $D^{M^+}_P$ and $D^{M^-}_{\id -P}$ be the Dirac operators equipped with $P$ and $\id -P_t$ respectively. Then, 
	\begin{equation}
		\tilde{\eta}_h(D^M)= \tilde{\eta}_h(D^{M^+}_P) + \tilde{\eta}_h(D^{M^-}_{\id -P}) - \spf_h\{D^{M^\pm}_{P_t \oplus \id -P_t}\}.
	\end{equation} 
\end{theorem}

\begin{remark}
It is noteworthy to point out that there was an attempt to give an alternate proof to this splitting formula of equivariant $\eta$-invariants by using a trick of Dai-Zhang \cite{10} by cutting a manifold into three parts. It relies on the behaviour of the invariants under a homotopy between two certain operators whose difference is a bounded operator. Unfortunately, $\eta$-invariants are not homotopy invariants in general. In special cases where this is true, another simplified proof of Theorem \ref{etasplitting1} can be obtained by using merely the properties of spectral flow, which does not require the angle-valued homotopy above.  
\end{remark}

The opposite symplectic structures between the two collar neighborhoods induce a flip of signs, e.g. $\gamma$ and $A$ over $\partial M^{+} \times (-\epsilon,0]$ becomes $-\gamma$ and $-A$ over $\partial M^{-} \times [0,\epsilon).$ Then, one can show that for compatibility reasons (for $P$ to satisfy the Lagrangian property and lies in $ \in \Gr_h(-A)$), the induced relation between the association map is given by $T_{-\gamma}(P)=-T_{\gamma}(\id -P)^*.$   
As a consequence, we have
\begin{equation}
	\label{eq:maslovopposymp}
	\mas^\gamma_h(\cL_1(t),\cL_2(t))=\mas^{-\gamma}_h(\cL_2(t),\cL_1(t)).
\end{equation}

\begin{theorem}
	\label{etasplitthm}
	With the above settings, for $h \in H,$ the following equality  holds
	\[	\tilde{\eta}_h(D^M)=\tilde{\eta}_h(D_P^{M^+}) + \tilde{\eta}_h(D_{\id-P}^{M^-}) +\tau^h_\mu (\id-\cP_{M^-},P,\cP_{M^+}).
	\] 
\end{theorem}

\begin{proof}
	Let $P\in \Gr_h(A).$ Since $(\id -\cP_{M^-})-\cP_{M^+}$ is of trace class, it is compact and its equivariant Maslov triple index 
	$\tau^h_\mu (\id-\cP_{M^-},P,\cP_{M^+})$ is well-defined for $P.$ 
	For $t\in [0,1],$ consider a smooth equivariant path $P(t)$ in $\Gr_h(A(t))$ connecting $\cP_{M^+}$ and $P.$ Let $\mathring{\cL}$ be the Lagrangian subspace corresponding to the orthogonal projection $\id -P_{\cL}.$
	
	Note that $(\mathring{\cL}_{M^-} (t),\cL_{M^+}(t))$ is constant for all $t,$ its equivariant Maslov index vanishes, i.e. $\mas^\gamma_h(\mathring{\cL}_{M^-} (t),\cL_{M^+}(t))=0.$ Then, by \eqref{eq:maslovopposymp}, we have
	\begin{align}
		\label{eq:mastau1}
		&\mas^{-\gamma}_h(\mathring{\cL}(t),\cL_{M^-}(t)) + \mas^\gamma_h(\cL(t),\cL_{M^+}(t)) \\
		&= \mas^{\gamma}_h(\mathring{\cL}_{M^-}(t),\cL(t)) + \mas^\gamma_h(\cL(t),\cL_{M^+}(t)) - \mas^\gamma_h(\mathring{\cL}_{M^-} (t),\cL_{M^+}(t)). \nonumber
	\end{align}
	This is an algebraic combination of three double Maslov indices, where the first, the second and the third paths of equivariant Lagrangians are $\mathring{\cL}_{M^-}(t), \cL(t)$ and $\cL_{M^+}(t)$ respectively.  
	By \eqref{eq:maslovtriple3}, we have 
	\begin{align}
		\label{eq:mastau2}
		&\mas^{\gamma}_h(\mathring{\cL}_{M^-}(t),\cL(t)) + \mas^\gamma_h(\cL(t),\cL_{M^+}(t))  - \mas^\gamma_h(\mathring{\cL}_{M^-} (t),\cL_{M^+}(t)) \\
		&=  \tau^h_\mu(\id - \cP_{M^-}, P, \cP_{M^+}) -  \tau^h_\mu(\id - \cP_{M^-}, \cP_{M^+}, \cP_{M^+}) = \tau^h_\mu(\id - \cP_{M^-}, P, \cP_{M^+}). \nonumber
	\end{align} 
	Finally, we obtain 
	\begin{align*}
		\tilde{\eta}_h(D^M)
		&= \tilde{\eta}_h(D_P^{M^+}) + \tilde{\eta}_h(D_{1-P}^{M^-}) +\spf_h\{D^M_{P}(t)\} + \spf_h\{D^{M^-}_{\id -P}(t)\}  \\
		&= \tilde{\eta}_h(D_P^{M^+}) + \tilde{\eta}_h(D_{1-P}^{M^-}) +\mas^\gamma_h(\cL(t),\cL_{M^+}(t)) + \mas^{-\gamma}_h(\mathring{\cL}(t),\cL_{M^-}(t))  \\
		&= \tilde{\eta}_h(D_P^{M^+}) + \tilde{\eta}_h(D_{1-P}^{M^-}) + \tau^h_{\mu}(\id - \cP_{M^-}, P, \cP_{M^+})
	\end{align*} 
	where the second equality follows from Theorem \ref{specmas} and the third equality follows from \eqref{eq:mastau1} and \eqref{eq:mastau2}. 
\end{proof}

\begin{remark}
A consequence of Theorem \ref{etasplitthm} will be illustrated in forming the boundary correction terms of the equivariant Toeplitz index theorem on odd-dimensional manifolds with boundary in \cite{16}. 
\end{remark}

\begin{example} (The following example and setting are adopted  from \cite[\S 8.1]{15} with proofs omitted.)
Let $X^{2n+1}$ be a compact manifold with non-empty boundary $Y^{2n} = \partial X$. Consider a flat $U(n)$-bundle $E \to X$ associated to a representation $\alpha: \pi_1(X) \to U(n),$ equipped with a flat connection $B.$ Suppose a collar of $Y$ is of product type and the boundary data are compatible, e.g. there is a flat connection $b$ on $E|_{\partial X}$ such that $B|_{[0,\epsilon) \times Y} = \mathrm{pr}_2^*(b)$ where $\mathrm{pr}_2:[0,\epsilon) \times Y \to Y$ is the projection map. 
Suppose that $H$ is a subgroup of the isometry group of $Y$ and the action of $H$ extends to a metric-preserving action on $X.$ 
Assume the action of $H$ lifts to an equivariant action on the vector bundle $E \to X$ preserving the inner product. The fixed point set of every $h \in H$ is the disjoint union of compact submanifolds $N,$ possibly with boundary $\partial N \subset Y.$ 
Let $d_B:\Omega^{p}(X,E) \to \Omega^{p+1}(X,E)$ be the exterior covariant derivative and $\ast: \Omega^{p}(X,E) \to \Omega^{2n+1-p}(X,E)$ be the Hodge $\ast$-operator. 
Consider the odd signature operator   $D_B: \oplus_{p}\Omega^{2p}(X,E) \to \oplus_{p}\Omega^{2p}(X,E),$ defined by 
\[
D_B(\xi) = i^{n+1}(-1)^{p-1}(\ast d_B - d_B \ast)(\xi), \quad \xi \in \Omega^{2p}(X,E).
\] 
Let $\hat{\ast}$ be the restriction on $Y$ and define $\gamma: \oplus_k \Omega^{k}(Y,E|_{Y}) \to \oplus_k \Omega^{k}(Y,E|_{Y})$ by 
$\gamma(\xi)= i^{n+1}(-1)^{p-1}\hat{\ast} \xi$ if $\xi \in \Omega^{2p}(Y,E|_{Y})$ and  $\gamma(\xi)= i^{n+1}(-1)^{n-p}\hat{\ast} \xi$ if $\xi \in \Omega^{2p+1}(Y,E|_{Y}).$  
Let $A_b: \oplus_k \Omega^{k}(Y, E|_{Y}) \to \oplus_k \Omega^{k}(Y, E|_{Y})$ be the tangential operator (associated to $b$) defined by $A_b=\pm(d_b \hat{\ast} + \hat{\ast}d_b)$ with $+$ (resp. $-$) corresponds to the one that acts on even (resp. odd) forms with values in $E|_{Y}.$ These $\gamma$ and $A_b$ satisfy \eqref{eq:gamma1} and \eqref{eq:gamma2}, cf. \cite[pg 41]{15}.  Moreover, there is an identification $\Phi: \oplus_p \Omega^{2p}([0,\epsilon) \times Y, E|_{[0,\epsilon) \times Y}) \to C^{\infty}([0,\epsilon), \oplus_k \Omega^k(Y,E|_Y))$ such that $\Phi D_B \Phi^* = \gamma(\partial_x + A_b)$ over the collar, taking the form \eqref{eq:dirac1}. 

By equivariance, every $h \in H$ commutes with $\gamma, A_b$ and thus $\Phi D_B \Phi^*$ as $h$ acts on $Y$ and trivially on $I$ over $Y \times I.$ 
By \cite[Lemma 8.6, eq (8.9)]{15} and by equivariance, there exists an $h$-invariant Lagrangian subspace $V_\alpha$ of $\ker(A)$ that corresponds to the eigenvalue $0$ of $A$. 
Following \cite[pp 609]{15}, with a representation $\alpha,$ a Lagragian subspace $K \subset H^*(Y,\C^n_\alpha),$ and the positive eigenspan $F^+_0$ of $A_b,$ the ordinary invariant $\eta(D,X,F^+_0 \oplus K)$ is well-defined. 
By modifying the methods \cite[pp 598-614]{15} appropriately, with $M=M^+ \cup_Y M^-$ joined along $Y,$ Theorem~\ref{etasplitthm} gives 
\[
\tilde{\eta}_h(D^M) = \tilde{\eta}_h(D_P^{M^+}; F^+_0 \oplus V) + \tilde{\eta}_h(D_{\id-P}^{M^-}; F^-_0 \oplus \gamma(V)) +\tau^h_\mu (\gamma(V^-_\alpha),P(V),V^+_\alpha)
\]
for $D$ the odd signature operator coupled with a flat connection; any Lagrangian subspace $V \subset \ker(A);$ and $V^\pm_\alpha \subset \ker(A)$ are isomorphic to $\im(H^{\text{ev}}(M^\mp;\C^n_\alpha) \to H^{\text{ev}}(N;\C^n_\alpha)),$ cf. \cite[Theorem 8.8]{15}. 

\end{example}

\section*{Acknowledgments}
Johnny Lim acknowledges the support from the Ministry of Higher Education Malaysia for Fundamental Research Grant Scheme with Project Code: 
\linebreak FRGS/1/2021/STG06/USM/02/7. Part of the work was done while the first author visited the Research Center for Operator Algebras at East China Normal University
back in October 2019. He is grateful to the center for financial support. Hang Wang is supported by NSFC 12271165, 23JC1401900, and in part by Science and
Technology Commission of Shanghai Municipality (No. 22DZ2229014).
We are grateful to the anonymous referees for their useful suggestions and improvement.



\end{document}